\documentclass{amsart}
\pdfoutput=1
\usepackage{amssymb}
\usepackage{thmtools, thm-restate}
\usepackage{hyperref} 

\declaretheorem[numberwithin=section]{theorem}
\declaretheorem[numberlike=theorem]{lemma}

\declaretheorem[numberlike=theorem]{corollary}

\declaretheorem[numberlike=theorem, style=definition]{definition}

\declaretheorem[numberlike=theorem, style=remark]{remark}

\numberwithin{equation}{section}

\title{Monotone {H}urwitz numbers and the {HCIZ} integral {II}}
\date{\today}

\author{I. P. Goulden}
\address{Department of Combinatorics \& Optimization. University of Waterloo, Canada}
\email{ipgoulden@uwaterloo.ca}

\author{M. Guay-Paquet}
\address{Department of Combinatorics \& Optimization. University of Waterloo, Canada}
\email{mguaypaq@uwaterloo.ca}
\thanks{IPG and MG-P were supported by NSERC}

\author{J. Novak}
\address{Department of Combinatorics \& Optimization. University of Waterloo, Canada}
\email{j2novak@uwaterloo.ca}

\date{\today}

\newcommand{\QQ}{\mathbb{Q}}

\newcommand{\id}{\mathrm{id}}
\renewcommand{\P}{\mathbb{P}}

\renewcommand{\S}{\mathbf{S}}

\newcommand{\cm}{{\mathcal{M}}}
\newcommand{\cmbar}{\overline{\cm}}

\renewcommand{\cm}{{\mathcal{M}}}
\renewcommand{\cmbar}{\overline{\cm}}

\newcommand{\Aut}{\operatorname{Aut}}
\newcommand{\hur}{H}
\newcommand{\Hur}{\mathbf{H}}
\newcommand{\mon}{\vec{H}}
\newcommand{\Mon}{\vec{\mathbf{H}}}
\newcommand{\Gen}{\mathbf{G}}
\newcommand{\C}{C}

\newcommand{\tr}{\operatorname{tr}}

\newcommand{\D}{\mathrm{D}}
\newcommand{\E}{\mathrm{E}}
\newcommand{\DD}{\mathcal{D}}
\newcommand{\EE}{\mathcal{E}}

\newcommand{\bigoh}{\operatorname{O}(p_i p_j)}

\newcommand{\abs}[1]{\left|{#1}\right|}
\newcommand{\diff}[2][]{\frac{\partial{#1}}{\partial{#2}}}
\newcommand{\sdiff}[3][]{\frac{\partial^2{#1}}{\partial{#2}\partial{#3}}}

\DeclareMathOperator*{\Split}{Split}
\DeclareMathOperator{\T}{T}

\begin{document}

\begin{abstract}
 	Motivated by results for the HCIZ integral in Part I of this paper, we study the structure of monotone Hurwitz numbers, which are a desymmetrized version of 
	classical Hurwitz numbers. We prove a number of results for monotone Hurwitz numbers and their generating series that are striking analogues of known results 
	for the classical Hurwtiz numbers. 
	These include explicit formulas for monotone Hurwitz numbers in genus $0$ and $1$, for all partitions, and an explicit rational form
	 for the generating series in arbitrary genus. This rational form implies that, up to an explicit combinatorial scaling, monotone Hurwitz numbers are polynomial in
	  the parts of the partition.
\end{abstract}

\maketitle

\setcounter{tocdepth}{2}
\tableofcontents

\setcounter{section}{-1}

\section{Introduction}\label{sec:intro}
This paper is a continuation of \cite{GGN}.  In \cite{GGN}, we studied the $N \rightarrow \infty$ asymptotics
of the Harish-Chandra-Itzykson-Zuber matrix model on the group of $N \times N$ unitary matrices, and showed that
the free energy of this matrix model admits an asymptotic expansion in powers of $N^{-2}$ whose coefficients are generating functions for a 
desymmetrized version of the double Hurwitz numbers \cite{GJV,O} which we called the \emph{monotone double
Hurwitz numbers}.  The monotone double Hurwitz number $\mon_g(\alpha,\beta)$ counts a combinatorially restricted subclass of the degree $d$
branched covers $f:C \rightarrow \P^1$ of the Riemann sphere by curves of genus $g$ which have ramification type $\alpha \vdash d$ over $\infty,$
$\beta \vdash d$ over $0,$ and $r=2g-2+\ell(\alpha)+\ell(\beta)$ additional simple branch points at fixed positions on $\P^1,$ the number of which 
is determined by the Riemann-Hurwitz formula.
The results of \cite{GGN} thus prove the existence of an asymptotic expansion in the HCIZ matrix model and
provide a topological interpretation of this expansion, thereby placing the HCIZ model on similar footing with the more developed
theory of topological expansion in Hermitian matrix models \cite{BIZ,BD,BI,EM,G:rigorous}.  See \cite{C,CGM} for previous results in this direction.

In this article, we give a thorough combinatorial analysis of the \emph{monotone single Hurwitz numbers} $\mon_g(\alpha)=\mon_g(\alpha,(1^d)),$ which count 
branched covers as above which are unramified over $0 \in \P^1.$
As explained in our first paper \cite{GGN}, the fixed-genus
generating functions of the monotone single Hurwitz numbers arise as the orders of the genus expansion in the ``one-sided'' HCIZ model.  The one-sided
HCIZ model is obtained when one
of the two sequences of normal matrices which define the HCIZ potential has degenerate limiting moments.\footnote{Note that this cannot happen if one
restricts to potentials defined by Hermitian matrices, since degeneracy would then violate the Hamburger moment criterion.}

Our study of the monotone single Hurwitz numbers $\mon_g(\alpha)$ is motivated by the following result, which is a degeneration
of the main theorem in our first paper \cite[Theorem 0.1]{GGN}.

\begin{theorem}
	\label{thm:Degenerate}
	Let $(A_N),(B_N)$ be two sequences of $N \times N$ normal matrices whose spectral radii are uniformly bounded, with least upper bound
	
		\begin{equation*}
			M:=\sup \ \{\rho(A_N),\rho(B_N) : N \geq 1\},
		\end{equation*}
	
	\noindent
	and which admit limiting moments
		
		\begin{align*}
			-\phi_k &:= \lim_{N \rightarrow \infty} \frac{1}{N} \tr(A_N^k) \\
			-\psi_k &:= \lim_{N \rightarrow \infty} \frac{1}{N} \tr(B_N^k) 
		\end{align*}
		
	\noindent
	of all orders.  Suppose furthermore that the limiting moments of $B_N$ are degenerate: $\psi_k=\delta_{1k}.$
	Let $0 \leq r < r_c,$ where $r_c$ is the critical value
	
		\begin{equation*}
			r_c = \frac{2}{27}.
		\end{equation*}
	
	\noindent
	Then, 
	the free energy $F_N(z)$ of the HCIZ model with potential $V=zN\tr(A_NUB_NU^*)$
	admits an $N \rightarrow \infty$ asymptotic expansion of the form
		\begin{equation*}
			F_N(z) \sim \sum_{g=0}^{\infty} \frac{C_g(z)}{N^{2g}}
		\end{equation*}
		
	\noindent
	which holds uniformly on the closed disc $\overline{D}(0,rM^{-2}).$  Each coefficient $C_g(z)$ is a holomorphic function of $z$ on the open disc
	$D(0,r_cM^{-2}),$ with Maclaurin series
	
		\begin{equation*}
			C_g(z) = \sum_{d=1}^{\infty} C_{g,d} \frac{z^d}{d!},
		\end{equation*}
		
	\noindent
	where 
	
		\begin{equation*}
			C_{g,d} = \sum_{\alpha \vdash d} \mon_g(\alpha) \phi_\alpha
		\end{equation*}
	
	\noindent	
	and $\mon_g(\alpha)$ is the number of $(r+1)$-tuples 
	$(\sigma,\tau_1,\dots,\tau_r)$ of permutations from the symmetric group $\S(d)$ such that
	
		\begin{enumerate}
			
			\medskip
			\item
			$\sigma$ has cycle type $\alpha$ and the $\tau_i$ are transpositions;
			
			\medskip
			\item
			The product $\sigma\tau_1 \dots \tau_r$ equals the identity permutation;
			
			\medskip
			\item
			The group $\langle \sigma,\tau_1,\dots,\tau_r \rangle$ acts transitively on $\{1,\dots,d\}$;
			
			\medskip
			\item
			$r=2g-2+\ell(\alpha)+d$;
			
			\medskip
			\item
			Writing $\tau_i=(s_i\ t_i)$ with $s_i<t_i,$ we have $t_1 \leq \dots \leq t_r.$
			
		\end{enumerate}
\end{theorem}

Conditions $(1)-(5)$ above may be taken as the definition of the monotone single Hurwitz numbers $\mon_g(\alpha);$ note that they 
differ from the classical single Hurwitz numbers $\hur_g(\alpha)$ only in the constraint imposed by Condition $(5)$\footnote{
Note that the usual geometric definition of the Hurwitz numbers contains a further division by $d!,$ which is omitted here.}.
According to Theorem \ref{thm:Degenerate}, the coefficients $C_g(z)$ in the $N \rightarrow \infty$ asymptotic expansion of the one-sided
HCIZ free energy are generating functions for the monotone single Hurwitz numbers in fixed genus
and all degrees.  In this article, we study the Witten-type formal generating series of the monotone single Hurwitz numbers
in all degrees and genera, \emph{i.e.} we study the limit object associated to the one-sided HCIZ free energy directly.

\subsection{Main results and organization}
In this paper we study the structure of monotone Hurwitz numbers, and focus in particular on the striking
similarities with classical Hurwitz numbers, which are present in almost every aspect of the theory.
The classical Hurwitz numbers \cite{H1} have enjoyed renewed interest since emerging as central objects in recent approaches to Witten's
conjecture \cite{W}.
Our main results are stated without proof in this section.  They are interleaved with the corresponding results
in the theory of classical Hurwitz numbers in order to emphasize these similarities.

Introduce the generating function

	\begin{equation}
		\Mon(z,t,p_1,p_2,\dots) = \sum_{d=1}^\infty \frac{z^d}{d!} \sum_{r=0}^{\infty} t^r \sum_{\alpha \vdash d} \mon^r(\alpha)p_\alpha,
	\end{equation}
	
\noindent
where $\mon^r(\alpha)=\mon_g(\alpha)$ with $r=2g-2+\ell(\alpha)+d$ and
$z,t,p_1,p_2,\dots$ are indeterminates. 
In Section \ref{sec:joincut}, we provide a 
global characterization of $\Mon$ in a manner akin to Virasoro constraints in random matrix theory: the monotone join-cut equation.

\begin{theorem}
	\label{thm:JoinCut}
  The generating function $\Mon$ is the unique formal power series solution of the 
  partial differential equation
  
  \begin{equation*}
   \frac{1}{2t}\bigg{(} z\diff[\Mon]{z} - z p_1 \bigg{)} = \frac{1}{2} \sum_{i,j \geq 1} (i+j)p_i p_j \diff[\Mon]{p_{i+j}} + ij p_{i+j} \sdiff[\Mon]{p_i}{p_j} + ij p_{i+j} \diff[\Mon]{p_i} \diff[\Mon]{p_j}
  \end{equation*}
  
  \noindent
 with the initial condition $[z^0] \Mon = 0$.
\end{theorem}

Note that this is almost exactly the same as the classical join-cut equation \cite{GJ:Hurwitz,GJVainshtein}
	\begin{equation}
		\diff[\Hur]{t}= \frac{1}{2} \sum_{i,j \geq 1} (i+j)p_i p_j \diff[\Hur]{p_{i+j}} + ij p_{i+j} \sdiff[\Hur]{p_i}{p_j} + ij p_{i+j} 
		 \diff	[\Hur]{p_i} \diff[\Hur]{p_j}
	\end{equation}
	
\noindent
which, together with the initial condition $[t^0]\Hur = zp_1,$ characterizes the generating function 

	\begin{equation}
		\Hur(z,t,p_1,p_2,\dots) = \sum_{d=1}^\infty \frac{z^d}{d!} \sum_{r=0}^{\infty} \frac{t^r}{r!} \sum_{\alpha \vdash d} \hur^r(\alpha)p_\alpha
	\end{equation}

\noindent
of the classical Hurwitz numbers, the only difference being that the left-hand side is a divided difference rather than a derivative with respect to $t.$
This is a consequence of the fact that, in the monotone case, $t$ is an ordinary rather than exponential marker for the number $r$ of simple ramification points
since the transpositions $\tau_i$ must be ordered as in Condition (5) of Theorem \ref{thm:Degenerate}.



In Sections \ref{sec:genus-zero} and \ref{sec:higher-genus}, we obtain explicit formulas for the low genus cases $\mon_0(\alpha)$ and $\mon_1(\alpha).$  The first of these is as follows.

\begin{theorem}
	\label{thm:gzformula}
  The genus zero monotone single Hurwitz numbers $\mon_0(\alpha),$ $\alpha \vdash d$ are given by
  \begin{equation*}
   \frac{|\Aut \alpha|}{d!}\mon_0(\alpha) = \bigg{(} \prod_{i=1}^{\ell(\alpha)} {2\alpha_i \choose \alpha_i} \bigg{)}(2d + 1)^{\overline{\ell(\alpha)-3}},
  \end{equation*}
   where
  \begin{equation*}(2d + 1)^{\overline{k}} = (2d + 1) (2d + 2) \cdots (2d + k)\end{equation*}
  denotes a rising product with $k$ factors, and
  by convention 
  
  	\begin{equation*}
		(2d + 1)^{\overline{k}} = \frac{1}{(2d+k+1)^{\overline{-k}}}
	\end{equation*}
  
  \noindent
  for $k<0.$
\end{theorem}

\noindent
In the extremal cases $\alpha=(1^d)$ and $\alpha=(d),$ this result was previously obtained by Zinn-Justin \cite{Z} and Gewurz and Merola \cite{GM}, respectively.
Theorem \ref{thm:gzformula} should be compared with the well-known explicit formula for the genus zero Hurwitz numbers

	\begin{equation}
		\frac{|\Aut \alpha|}{d!}\hur_0(\alpha) = (d-2+\ell(\alpha))! \bigg{(} \prod_{i=1}^{\ell(\alpha)} \frac{\alpha_i^{\alpha_i}}{\alpha_1!} \bigg{)}
		d^{\ell(\alpha)-3} 
	\end{equation}
published without proof by Hurwitz \cite{H1} in 1891 and independently rediscovered and proved a century later by Goulden and 
Jackson \cite{GJ:Hurwitz}.  

The explicit formula for genus one monotone Hurwitz numbers is as follows.

\begin{theorem}\label{thm:goformula}
  The genus one monotone single Hurwitz numbers $\mon_1(\alpha),$ $\alpha \vdash d$ are given by
  \begin{equation*}
  	\begin{split}
    & \frac{|\Aut \alpha|}{d!}\mon_1(\alpha) = \frac{1}{24}\prod_{i=1}^{\ell(\alpha)} {2\alpha_i \choose \alpha_i} \\
   & \times \bigg{(}(2d + 1)^{\overline{\ell(\alpha)}} - 
    3 (2d + 1)^{\overline{\ell(\alpha) - 1}} - \sum_{k = 2}^{\ell(\alpha)} (k - 2)! (2d + 1)^{\overline{\ell(\alpha) - k}} e_k(2\alpha + 1)\bigg{)},
    \end{split}
  \end{equation*}
  where $e_k(2\alpha + 1)$ 
  is the $k$th elementary symmetric polynomial in $2\alpha_i+1,1 \leq i \leq \ell(\alpha).$
\end{theorem}

\noindent
This result should be compared with
the explicit formula for the genus one classical Hurwitz numbers $H_1(\alpha)$, 

	\begin{equation}
		\begin{split}
		&\frac{|\Aut \alpha|}{d!} \hur_1(\alpha)= \frac{(d+\ell(\alpha))!}{24} \prod_{i=1}^{\ell(\alpha)} \frac{\alpha_i^{\alpha_i}}{\alpha_i!} \\
		&\times \left( d^{\ell (\alpha)} -d^{\ell (\alpha)-1} -\sum_{k=2}^{\ell (\alpha)} (k-2)! d^{\ell (\alpha) -k} e_k(\alpha ) \right),
		\end{split}
	\end{equation}

\noindent
which was conjectured in \cite{GJVainshtein} 
and proved by Vakil \cite{V}, see also \cite{GJ:torus}.  

Via Theorem \ref{thm:Degenerate},
theorems \ref{thm:gzformula} and \ref{thm:goformula} yield explicit forms for the first two orders in the free energy 
of the one-sided HCIZ model.  These results may be compared with the first two orders in the free energy of the Hermitian
one-matrix model, which for example in the case of cubic vertices were conjectured by Br\'ezin, Itzykson, Parisi and Zuber in \cite{BIPZ}
and rigorously verified in \cite{BD}.

In Section \ref{sec:higher-genus}, we obtain explicit forms for the fixed-genus generating functions 

	\begin{equation}
		\Mon_g(z,p_1,p_2,\dots) = \sum_{d=1}^{\infty} \mon_g(\alpha)p_\alpha \frac{z^d}{d!}
	\end{equation}

\noindent
for the monotone single Hurwitz numbers in terms of an implicit set of Lagrangian variables.

\begin{theorem}\label{thm:goseries}
  Let $s$ be the unique formal power series solution of the functional equation
  \begin{equation*}
    s = z \left( 1 -\gamma\right)^{-2}
  \end{equation*}
  in the ring $\QQ[[z,p_1,p_2,\dots]]$, where $\gamma = \sum_{k \geq 1} \binom{2k}{k} p_k s^k$. 
  Also, define $\eta = \sum_{k \geq 1} (2k + 1) \binom{2k}{k}p_k s^k$.  Then,
  the genus one monotone single Hurwitz generating series is
  \begin{equation*}
    \Mon_1(z,p_1,p_2,\dots) = \tfrac{1}{24} \log \frac{1}{1 - \eta} - \tfrac{1}{8} \log \frac{1}{1-\gamma} 
    =\tfrac{1}{24} \log \bigg{(} \bigg{(} \frac{z}{s} \bigg{)}^2 \frac{\partial s}{\partial z} \bigg{)}\bigg{)},
  \end{equation*}
  and for $g \geq 2$ we have
  
  	\begin{equation*}
    \Mon_g(z,p_1,p_2,\dots) = -c_{g,(0)}+
  \frac{1}{(1 - \eta)^{2g - 2}} \sum_{d = 0}^{3g - 3} \sum_{\alpha \vdash d} \frac{c_{g,\alpha} \eta_\alpha}{(1 - \eta)^{\ell(\alpha)}},
  \end{equation*}
 
 \noindent
 where 
 
 \begin{equation*}
    \eta_j = \sum_{k \geq 1} k^j (2k + 1) \binom{2k}{k} p_k s^k, \quad j \geq 1,
  \end{equation*}  
  
  \noindent
  and the 
 $c_{g,\alpha}$ are rational constants.  In particular, for the empty partition,
 
 \begin{equation*}
 	c_{g,(0)}= -\frac{B_{2g}}{4g(g-1)}
 \end{equation*}
 
\noindent
where $B_{2g}$ is a Bernoulli number.
  
\end{theorem}

These explicit forms for $\Mon_g$ should be compared with the analogous explicit forms for the 
generating series

	\begin{equation}
		\Hur_g(z,p_1,p_2,\dots) = \sum_{d=1}^{\infty} \sum_{\alpha \vdash d} \frac{\hur_g(\alpha)}{(2g-2+\ell(\alpha)+d)!}p_\alpha \frac{z^d}{d!}
	\end{equation}
	
\noindent
for the classical single Hurwitz numbers.
Adapting notation from previous works \cite{GJ:torus, GJ:rationality, GJV:GromovWitten} in order to highlight this analogy, 
 let $w$ be the unique formal power series solution of the functional equation
  \begin{equation*}
    w = z e^{\delta}
  \end{equation*}
  in the ring $\QQ[[z,p_1,p_2,\dots]]$, where $\delta = \sum_{k \geq 1} \frac{k^k}{k!} p_k w^k$. 
  Also, define $\phi = \sum_{k \geq 1}  \frac{k^{k+1}}{k!} p_k w^k$.  Then,
  the genus one single Hurwitz generating series is~\cite{GJ:torus}
  \begin{equation*}
    \Hur_1(z,p_1,p_2,\dots) = \tfrac{1}{24} \log \frac{1}{1 - \phi} - \tfrac{1}{24}\delta
    =\tfrac{1}{24} \log \bigg{(} \bigg{(} \frac{z}{w} \bigg{)}^2 \frac{\partial w}{\partial z} \bigg{)}\bigg{)},
  \end{equation*}
  
  \noindent
  and for $g \geq 2$ we have \cite{GJV:GromovWitten}
  
  	\begin{equation}
		\label{eqn:HurwitzRational}
    \Hur_g(z,p_1,p_2,\dots) = 
  \frac{1}{(1 - \phi)^{2g - 2}} \sum_{d = 1}^{3g - 3} \sum_{\alpha \vdash d} \frac{a_{g,\alpha} \phi_{\alpha}}{(1 - \phi)^{\ell(\alpha)}},
  \end{equation}
 
 \noindent
 where 
 
 \begin{equation*}
    \phi_j = \sum_{k \geq 1} \frac{k^{k+j+1}}{k!}p_k w^k, \quad j \geq 1,
  \end{equation*}  
  
  \noindent
  and the 
 $a_{g,\alpha}$ are rational constants.
 
 For genus $g=2,3$, these rational forms are given in the Appendix of this paper.  The corresponding
 rational forms for the classical Hurwitz series can be found in \cite{GJ:rationality}.
 Comparing these expressions, one may observe that in all cases
 
 	\begin{equation}
		\label{eqn:PowerScaling}
 	c_{g,\alpha} = 2^{3g-3} a_{g,\alpha}, \quad \alpha\vdash 3g-3,
 	\end{equation}
 
 \noindent	
 for $g=2,3$. 
  
A key consequence of Theorem \ref{thm:goseries}, also proved in Section \ref{sec:higher-genus}, is that it implies the polynomiality of the monotone 
single Hurwitz numbers themselves.

\begin{theorem}\label{thm:polynomial}
 To each pair $(g,m)$ with $(g,m) \notin \{(0,1), (0,2)\}$ there corresponds a polynomial $\vec{P}_g$ in $m$ variables such that
  \begin{equation*}
    \frac{|\Aut \alpha|}{|\alpha|!}\mon_g(\alpha) = \bigg{(}\prod_{i=1}^m {2\alpha_i \choose \alpha_i}\bigg{)} \vec{P}_g(\alpha_1,\dots,\alpha_m)
  \end{equation*}
  for all partitions $\alpha$ with $\ell(\alpha)=m.$
\end{theorem}

\noindent
Theorem \ref{thm:polynomial} is the exact analogue of polynomiality, originally conjectured in \cite{GJVainshtein}, for the classical Hurwitz numbers, which
under the same hypotheses as in Theorem \ref{thm:polynomial} asserts the existence of polynomials $P_g$ in $m$ variables
such that
	
	\begin{equation}
		\frac{|\Aut \alpha|}{|\alpha|!}\hur_g(\alpha) = (2g-2+m+|\alpha|)! \bigg{(}\prod_{i=1}^m \frac{\alpha_i^{\alpha_i}}{\alpha_i!}
		\bigg{)} P_g(\alpha_1,\dots,\alpha_m)
	\end{equation}
	
\noindent
for all partitions $\alpha$ with $m$ parts.  The only known proof of this result relies on the ELSV formula \cite{ELSV}

	\begin{equation}\label{elsvf}
		P_{g}(\alpha_1,\dots,\alpha_m) = \int_{\cmbar_{g,m}} \frac {1 - \lambda_1 + \cdots + (-1)^g \lambda_g}{ (1 - \alpha_1 \psi_1) \cdots (1 - \alpha_m \psi_m)}.
	\end{equation}
	
\noindent
Here $\cmbar_{g,m}$ is the (compact) moduli space of stable
$m$-pointed genus $g$ curves, $\psi_1$, $\dots,$ $\psi_m$ are
(complex) codimension $1$ classes corresponding to the $m$ marked
points, and $\lambda_k$ is the (complex codimension $k$) $k$th Chern class
of the Hodge bundle.
Equation~\eqref{elsvf} should be interpreted as follows:
formally invert the denominator as a geometric series;
select the terms of codimension~$\dim \cmbar_{g,m}=3g-3+m$;
and ``intersect'' these terms on $\cmbar_{g,m}$.
In contrast to this, our proof of Theorem \ref{thm:polynomial} is entirely algebraic and makes no use of
geometric methods.

A geometric approach to the monotone Hurwitz numbers would be highly desirable.
The form of the rational expression given in Theorem \ref{thm:goseries}, in particular its high degree of similarity with the corresponding
rational expression \eqref{eqn:HurwitzRational} for the generating series of the classical Hurwitz numbers, suggests the possibility of an ELSV-type formula
for the polynomials $\vec{P}_g.$  Further evidence in favour of such a formula is that the summation sets differ only by a contribution
from the empty partition, which is itself a scaled Bernoulli number, of known geometric significance.   Finally, observe that the ELSV formula
implies that the coefficients $a_{g,\alpha}$ in the rational form \eqref{eqn:HurwitzRational} are themselves Hodge integral evaluations,
and for the top terms $\alpha \vdash 3g-3$ these Hodge integrals are free of $\lambda$-classes --- the Witten case.  
Equation~\eqref{eqn:PowerScaling}, which deals precisely with the case $\alpha \vdash 3g-3$, might be a good starting point for the formulation 
of such a geometric result.

Section \ref{sec:joincut} closes with two results left unstated here, since they are of a more technical nature than those summarized above. 
These are a topological recursion
in the style of Eynard and Orantin \cite{EO}, and a join-cut equation for the monotone double Hurwitz series.  Unlike the classical case \cite{GJV},
the join-cut equation for the monotone double Hurwitz numbers does \emph{not} coincide with the join-cut equation for the single 
monotone Hurwitz numbers.


\subsection{Acknowledgements}
	It is a pleasure to acknowledge helpful conversations with our colleagues Sean Carrell and David Jackson, Waterloo, 
	and Ravi Vakil, Stanford.  J. N. would like to acknowledge email correspondence with Mike Roth, Queen's.
	The extensive numerical computations required in this project were performed using Sage, and
	its algebraic combinatorics features developed by the Sage-Combinat community.

\section{Join-cut analysis}\label{sec:joincut}


In this section, we analyse the effect of removing the last factor in a transitive monotone factorization. From this, we obtain a recurrence relation for the number of these factorizations, and differential equations which characterize some related generating series.

\subsection{Recurrence relation}\label{sec:recurrence}

Let $M^r(\alpha)$ be defined by 

	\begin{equation}
		\mon^r(\alpha) =|C_\alpha| M^r(\alpha).
	\end{equation}
	
\noindent
It follows from the centrality of symmetric functions of Jucys-Murphy elements, see \cite{GGN}, that
$M^r(\alpha)$ counts the number of $(r+1)$-tuples $(\sigma_0,\tau_1,\dots,\tau_r)$ satisfying conditions $(1)-(5)$ of Theorem \ref{thm:Degenerate}, where
$\sigma_0$ is a fixed but arbitrary permutation of cycle type $\alpha.$

\begin{theorem}\label{thm:recur}
  The numbers $M^r(\alpha)$ are uniquely determined by the initial condition
  \begin{equation}
    M^0(\alpha) = \begin{cases}
      1 &\text{if $\alpha = \varepsilon$}, \\
      0 &\text{otherwise},
    \end{cases}
  \end{equation}
  and the recurrence
  \begin{multline}\label{recurrence}
    M^{r+1}(\alpha \cup \{k\})
      = \sum_{k' \geq 1} k' m_{k'}(\alpha) M^r(\alpha \setminus \{k'\} \cup \{k + k'\}) \\
      {} + \sum_{k' = 1}^{k - 1} M^r(\alpha \cup \{k', k - k'\}) \\
      {} + \sum_{k' = 1}^{k - 1} \sum_{r' = 0}^r \sum_{\alpha' \subseteq \alpha} M^{r'}(\alpha' \cup \{k'\}) M^{r - r'}(\alpha \setminus \alpha' \cup \{k - k'\}),
  \end{multline}
  where $m_{k'}(\alpha)$ is the multiplicity of $k'$ as a part in the partition $\alpha$,
  and the last sum is over the $2^{\ell(\alpha)}$ subpartitions $\alpha'$ of $\alpha$.
\end{theorem}

\begin{proof}
  As long as the initial condition and the recurrence relation hold, uniqueness follows by induction on $r$. The initial condition follows from the fact that for $r = 0$ we must have $\sigma = \id$, and the identity permutation is only transitive in $S_1$.
  
  To show the recurrence, fix a permutation $\sigma \in S_d$ of cycle type $\alpha \cup \{k\}$, where the element $d$ is in a cycle of length $k$, and consider a transitive monotone factorization
  \begin{equation}\label{fact1}
    (a_1 \, b_1) (a_2 \, b_2) \cdots (a_r \, b_r) (a_{r+1} \, b_{r+1}) = \sigma.
  \end{equation}
  The transitivity condition forces the element $d$ to appear in some transposition, and the monotonicity condition forces it to appear in the last transposition, so it must be that $b_{r+1} = d$. If we move this transposition to the other side of the equation and set $\sigma' = \sigma (a_{r+1} \, b_{r+1})$, we get the shorter monotone factorization
  \begin{equation}\label{fact2}
    (a_1 \, b_1) (a_2 \, b_2) \cdots (a_r \, b_r) = \sigma'.
  \end{equation}
  Depending on whether $a_{r+1}$ is in the same cycle of $\sigma'$ as $b_{r+1}$ and whether \eqref{fact2} is still transitive, the shorter factorization falls into exactly one of the following three cases, corresponding to the three terms on the right-hand side of the recurrence.
  \begin{description}
    \item[Cut]
      Suppose $a_{r+1}$ and $b_{r+1}$ are in the same cycle of $\sigma'$. Then, $\sigma$ is obtained from $\sigma'$ by cutting the cycle containing $a_{r+1}$ and $b_{r+1}$ in two parts, one containing $a_{r+1}$ and the other containing $b_{r+1}$, so $(a_{r+1} \, b_{r+1})$ is called a \emph{cut} for $\sigma'$, and also for the factorization \eqref{fact1}. Conversely, $a_{r+1}$ and $b_{r+1}$ are in different cycles of $\sigma$, and $\sigma'$ is obtained from $\sigma$ by joining these two cycles, so the transposition $(a_{r+1} \, b_{r+1})$ is called a \emph{join} for $\sigma$. Note that in the case of a cut, \eqref{fact2} is transitive if and only if \eqref{fact1} is transitive.
      
      For $k' \geq 1$, there are $k' m_{k'}(\alpha)$ possible choices for $a_{r+1}$ in a cycle of $\sigma$ of length $k'$ other than the one containing $b_{r+1}$. For each of these choices, $(a_{r+1} \, b_{r+1})$ is a cut and $\sigma'$ has cycle type $\alpha \setminus \{k'\} \cup \{k + k'\}$. Thus, the number of transitive monotone factorizations of $\sigma$ where the last factor is a cut is
      \begin{equation}
        \sum_{k' \geq 1} k' m_{k'}(\alpha) M^r(\alpha \setminus \{k'\} \cup \{k + k'\}),
      \end{equation}
      which is the first term in the recurrence.
    
    \item[Redundant join]
      Now suppose that $(a_{r+1} \, b_{r+1})$ is a join for $\sigma'$ and that \eqref{fact2} is transitive. Then, we say that $(a_{r+1} \, b_{r+1})$ is a \emph{redundant join} for \eqref{fact1}.
      
      The transposition $(a_{r+1} \, b_{r+1})$ is a join for $\sigma'$ if and only if it is a cut for $\sigma$, and there are $k - 1$ ways of cutting the $k$-cycle of $\sigma$ containing $b_{r+1}$. Thus, the number of transitive monotone factorizations of $\sigma$ where the last factor is a redundant join is
      \begin{equation}
        \sum_{k' = 1}^{k - 1} M^r(\alpha \cup \{k', k - k'\}),
      \end{equation}
      which is the second term in the recurrence.
    
    \item[Essential join]
      Finally, suppose that $(a_{r+1} \, b_{r+1})$ is a join for $\sigma'$ and that \eqref{fact2} is not transitive. Then, we say that $(a_{r+1} \, b_{r+1})$ is an \emph{essential join} for \eqref{fact1}. In this case, the action of the subgroup $\langle (a_1 \, b_1), \ldots, (a_r \, b_r) \rangle$ must have exactly two orbits on the ground set, one containing $a_{r+1}$ and the other containing $b_{r+1}$. Since transpositions acting on different orbits commute, \eqref{fact2} can be rearranged into a product of two transitive monotone factorizations on these orbits. Conversely, given a transitive monotone factorization for each orbit, this process can be reversed, and the monotonicity condition guarantees uniqueness of the result.
      
      As with redundant joins, there are $k - 1$ choices for $a_{r+1}$ to split the $k$-cycle of $\sigma$ containing $b_{r+1}$. Each of the other cycles of $\sigma$ must be in one of the two orbits, so there are $2^{\ell(\alpha)}$ choices for the orbit containing $a_{r+1}$. Thus, the number of transitive monotone factorizations of $\sigma$ where the last factor is an essential join is
      \begin{equation}
        \sum_{k' = 1}^{k - 1} \sum_{r' = 0}^r \sum_{\alpha' \subseteq \alpha} M^{r'}(\alpha' \cup \{k'\}) M^{r - r'}(\alpha \setminus \alpha' \cup \{k - k'\}),
      \end{equation}
      which is the third term in the recurrence.
      \qedhere
  \end{description}
\end{proof}

\subsection{Operators}

To write the recurrence from \autoref{thm:recur} as a differential equation for some generating series, we introduce some operators. They will be used in 
later sections to manipulate this equation and solve it.

\begin{definition}
  The three \emph{lifting} operators are the $\QQ[[x, y, z, p_1, p_2, \ldots]]$-linear differential operators
  \begin{align*}
    \Delta_x &= \sum_{k \geq 1} k x^k \diff{p_k}, &
    \Delta_y &= \sum_{k \geq 1} k y^k \diff{p_k}, &
    \Delta_z &= \sum_{k \geq 1} k z^k \diff{p_k}.
  \end{align*}
\end{definition}

The combinatorial effect of $\Delta_x$, when applied to a generating series, is to pick a cycle marked by $p_k$ in all possible ways and mark it by $k x^k$ instead, that is, by $x^k$ once for each element of the cycle. 
Note that $\Delta_x x = 0$, so that
\begin{equation}\label{doubledelta}
  \Delta_x^2 = \sum_{i,j \geq 1} ij x^{i+j} \sdiff{p_i}{p_j}
\end{equation}
These operators are called lifting operators because of the corresponding projection operators:

\begin{definition}
  The three \emph{projection} operators are the operators
  \begin{align*}
    \Pi_x &= [x^0] + \sum_{k \geq 1} p_k [x^k], &
    \Pi_y &= [y^0] + \sum_{k \geq 1} p_k [y^k], &
    \Pi_z &= [z^0] + \sum_{k \geq 1} p_k [z^k],
  \end{align*}
  where $\Pi_x$ is $\QQ[[y, z, p_1, p_2, \ldots]]$-linear, $\Pi_y$ is $\QQ[[x, z, p_1, p_2, \ldots]]$-linear, and $\Pi_z$ is $\QQ[[x, y, p_1, p_2, \ldots]]$-linear. These operators commute and are idempotent, and $\Pi_{xy} = \Pi_x \Pi_y$, $\Pi_{xyz} = \Pi_x \Pi_y \Pi_z$, etc.\ denote their compositions.
\end{definition}

The combined effect of a lift and a projection when applied to a generating series in $\QQ[[p_1, p_2, \ldots]]$ is given by
\begin{equation}\label{pidelta}
  \Pi_x \Delta_x = \Pi_y \Delta_y = \Pi_z \Delta_z = \sum_{k \geq 1} k p_k \diff{p_k}.
\end{equation}
The identity
\begin{equation}\label{deltapi}
  \Delta_x \Pi_y f = \left[ y\diff{y} f \right]_{y=x} + \Pi_y \Delta_x f,
\end{equation}
which holds for $f \in \QQ[[x, y, z, p_1, p_2, \ldots]]$, will be useful later on, and can be checked by verifying it on elements of the form $x^i y^j z^k p_\alpha$. We will also need some splitting operators.

\begin{definition}
  The \emph{splitting} operator is the $\QQ[[p_1, p_2, \ldots]]$-linear operator defined by
  \begin{equation*}
    \Split_{x \to y}(x^k) = x^{k-1} y + x^{k-2} y^2 + \cdots + x y^{k-1}
  \end{equation*}
  for $k \geq 2$, and by $\Split_{x \to y}(1) = \Split_{x \to y}(x) = 0$. If $f(x)$ is a power series in $x$ over $\QQ[[p_1, p_2, \ldots]]$ with no constant term, then
  \begin{equation}\Split_{x \to y} \big( f(x) \big) = \frac{y f(x) - x f(y)}{x - y}.\end{equation}
\end{definition}

The combined effect of a lift, a split and a projection on a generating series in $\QQ[[p_1, p_2, \ldots]]$ is
\begin{equation}\label{pisplit}
  \Pi_{xy} \Split_{x \to y} \Delta_x = \sum_{i,j \geq 1} (i + j) p_i p_j \diff{p_{i + j}}.
\end{equation}

\subsection{Differential equations}

With the lifting, projection and splitting operators, we can write the recurrence of \autoref{thm:recur} in terms of generating series.

\begin{theorem}\label{thm:diffeq1}
  For $f = f(z, t, x, p) \in \QQ[[z, t, x, p_1, p_2, \ldots]]$, the differential equation
  \begin{equation}\label{diffeq1}
    \frac{f - zx}{t} = \Pi_y \Split_{x \to y} f + \Delta_x f + f^2
  \end{equation}
  has the unique solution $f = \Delta_x \Mon(z, t, p)$.
\end{theorem}

\begin{proof}
  The coefficient $[t^r] f$ of a solution $f \in \QQ[[z, t, x, p_1, p_2, \ldots]]$ can be computed recursively from coefficients $[t^{r'}] f$ for $r' < r$, with the base case $[t^0] f = zx$. Thus, by induction on $r$, such a solution exists and is unique.
  
  To show that $\Delta_x \Mon$ is a solution, we use the fact that \eqref{diffeq1} is equivalent to the initial condition and recurrence from \autoref{thm:recur}. Indeed, multiplying \eqref{recurrence} by
  \begin{equation}
    \frac{z^d}{d!} t^r p_\alpha k (m_k(\alpha) + 1) x^k \abs{\C_{\alpha \cup \{k\}}} = \frac{z^d t^r p_\alpha x^k}{\prod_{i \geq 1} i^{m_i(\alpha)} m_i(\alpha)!}
  \end{equation}
  and summing over all choices of $d, r, \alpha, k$ with $d \geq k \geq 1$, $\alpha \vdash d - k$ and $r \geq 0$ gives
  \begin{equation}\label{diffeqsubst}
    \frac{\Delta_x \Mon - zx}{t} = \Pi_y \Split_{x \to y} \Delta_x \Mon + \Delta_x^2 \Mon + \big(\Delta_x \Mon)^2,
  \end{equation}
  as shown by the following computations. We have
  \begin{align}
    \Delta_x \Mon
      &= \sum_{\substack{k \geq 1 \\ d \geq 1 \\ r \geq 0 \\ \alpha \vdash d}} \frac{z^d}{d!} t^r \diff[p_\alpha]{p_k} k x^k \abs{\C_\alpha} M^r(\alpha) \\
      &= \sum_{\substack{k \geq 1 \\ d \geq 1 \\ r \geq 0 \\ \alpha \vdash d - k}} \frac{z^d}{d!} t^r p_\alpha k (m_k(\alpha) + 1) x^k \abs{\C_{\alpha \cup \{k\}}} M^r(\alpha \cup \{k\}) \\
      &= \sum_{\substack{k \geq 1 \\ d \geq 1 \\ r \geq 0 \\ \alpha \vdash d - k}} \frac{z^d t^r p_\alpha x^k M^r(\alpha \cup \{k\}}{\prod_{i \geq 1} i^{m_i(\alpha)} m_i(\alpha)!}),
  \end{align}
  so that
  \begin{equation}
    \frac{\Delta_x \Mon - zx}{t}
      = \sum_{\substack{k \geq 1 \\ d \geq k \\ r \geq 0 \\ \alpha \vdash d - k}} \frac{z^d t^r p_\alpha x^k M^{r+1}(\alpha \cup \{k\})}{\prod_{i \geq 1} i^{m_i(\alpha)} m_i(\alpha)!},
  \end{equation}
  which gives the left-hand side of \eqref{recurrence} and \eqref{diffeqsubst}. Next, we have
  \begin{align}
    \Pi_y \Split_{x \to y} \Delta_x \Mon
      &= \sum_{\substack{k' \geq 1 \\ k \geq k' + 1 \\ d \geq k \\ r \geq 0 \\ \alpha \vdash d - k}} \frac{z^d t^r p_{\alpha \cup \{k'\}} x^{k - k'} M^r(\alpha \cup \{k\})}{\prod_{i \geq 1} i^{m_i(\alpha)} m_i(\alpha)!} \\
      &= \sum_{\substack{k \geq 1 \\ d \geq k \\ r \geq 0 \\ \alpha \vdash d - k \\ k' \geq 1,\, k' \in \alpha}} \frac{z^d t^r p_\alpha k' m_{k'}(\alpha) x^k M^r(\alpha \setminus \{k'\} \cup \{k + k'\})}{\prod_{i \geq 1} i^{m_i(\alpha)} m_i(\alpha)!}, \\
  \end{align}
  where the second line is obtained from the first by reindexing, replacing $k$ by $k + k'$ and $\alpha$ by $\alpha \setminus \{k'\}$. This is the first term of the right-hand side of \eqref{recurrence} and \eqref{diffeqsubst}. Also, we have
  \begin{align}
    \Delta_x^2 \Mon
      &= \sum_{\substack{k' \geq 1 \\ k \geq 1 \\ d \geq k \\ r \geq 0 \\ \alpha \vdash d - k}} k' x^{k'} \diff{p_{k'}} \frac{z^d t^r p_\alpha x^k M^r(\alpha \cup \{k\})}{\prod_{i \geq 1} i^{m_i(\alpha)} m_i(\alpha)!} \\
      &= \sum_{\substack{k' \geq 1 \\ k \geq k + 1 \\ d \geq k \\ r \geq 0 \\ \alpha \vdash d - k}} \frac{z^d t^r p_\alpha x^k M^r(\alpha \cup \{k', k - k'\})}{\prod_{i \geq 1} i^{m_i(\alpha)} m_i(\alpha)!},
  \end{align}
  where the second line is obtained from the first by removing terms which vanish (that is, with $k' \notin \alpha$) and reindexing, replacing $k$ by $k - k'$ and $\alpha$ by $\alpha \cup \{k'\}$. This is  the second term of the right-hand side of \eqref{recurrence} and \eqref{diffeqsubst}. Finally, we have
  \begin{align}
    \big(\Delta_x \Mon)^2
      &= \sum_{\substack{k' \geq 1 \\ d' \geq k' \\ r' \geq 0 \\ \alpha' \vdash d' - k'}} \; \sum_{\substack{k \geq 1 \\ d \geq k \\ r \geq 0 \\ \alpha \vdash d - k}} \frac{z^{d + d'} t^{r + r'} p_{\alpha \cup \alpha'} x^{k + k'}  M^{r'}(\alpha' \cup \{k'\}) M^r(\alpha \cup \{k\})}{\prod_{i \geq 1} i^{m_i(\alpha) + m_i(\alpha')} m_i(\alpha)! m_i(\alpha')!} \\
      &= \sum_{\substack{k' \geq 1 \\ d' \geq k' \\ r' \geq 0 \\ \alpha' \vdash d' - k'}} \; \sum_{\substack{k \geq k' + 1 \\ d \geq k + d' - k' \\ r \geq r' \\ \alpha \vdash d - k,\, \alpha' \subseteq \alpha}} \frac{z^d t^r p_\alpha x^k M^{r'}(\alpha' \cup \{k'\}) M^{r - r'}(\alpha \setminus \alpha' \cup \{k - k'\})}{\prod_{i \geq 1} i^{m_i(\alpha)} m_i(\alpha)! \binom{m_i(\alpha)}{m_i(\alpha')}^{-1}} \\
      &= \sum_{\substack{k \geq 1 \\ d \geq k \\ r \geq 0 \\ \alpha \vdash d - k}} \; \sum_{\substack{k' \geq 1,\, k' \leq k - 1 \\ \alpha' \subseteq \alpha \\ r' \geq 0,\, r' \leq r}} \frac{z^d t^r p_\alpha x^k M^{r'}(\alpha' \cup \{k'\}) M^{r - r'}(\alpha \setminus \alpha' \cup \{k - k'\})}{\prod_{i \geq 1} i^{m_i(\alpha)} m_i(\alpha)!},
  \end{align}
  where the second line is obtained from the first by reindexing, replacing $d$ by $d - d'$, $k$ by $k - k'$ and $\alpha$ by $\alpha \setminus \alpha'$; and the third line is obtained by replacing the summation over $\alpha' \vdash d' - k'$ by a summation over the $2^{\ell(\alpha)}$ subpartitions of $\alpha$, weighted by $\prod_{i \geq 1} \binom{m_i(\alpha)}{m_i(\alpha')}^{-1}$ to account for the resulting overcount. This is the third term of the right-hand side of \eqref{recurrence} and \eqref{diffeqsubst}.
\end{proof}

As a corollary, we obtain a proof of the monotone join-cut equation.
\medskip

\noindent
\textit{Proof of Theorem \ref{thm:JoinCut}}.
  This equation is the result of applying the projection operator $\tfrac{1}{2} \Pi_x$ to \eqref{diffeqsubst}, and simplifying with \eqref{pidelta}, \eqref{pisplit}, \eqref{doubledelta}, and noting that
  \begin{equation}
    \sum_{k \geq 1} k p_k \diff[\Mon]{p_k} \Mon = z \diff[\Mon]{z},
  \end{equation}
  so it is satisfied by $\Mon$. Except for $[z^0] \Mon$, uniqueness of the coefficients of $\Mon$ follows by induction on the exponent of the accompanying 
  power of $t$. $\hfill\Box$

\medskip

In what follows, it will be convenient to specialize $z=1$ in our generating series for monotone Hurwitz numbers.

\begin{definition}\label{def:Gseries}
  For $g \geq 0$, the genus $g$ generating series for monotone single Hurwitz numbers is
  \begin{equation}
    \Gen_g = \Mon_g(1,p_1, p_2, \ldots) = \sum_{\substack{d \geq 1 \\ \alpha \vdash d}} \frac{p_\alpha}{d!} \mon_g(\alpha),
  \end{equation}
  and the genus-wise generating series for all monotone single Hurwitz numbers is
  \begin{equation}
    \Gen = \Gen(t, p_1, p_2, \ldots) = \sum_{g \geq 0} t^g \Gen_g.
  \end{equation}
\end{definition}

Note that this generating series is equivalent to $\Mon$, via the relations
\begin{align}\label{substitution}
  \Gen(t, p_1, p_2, \ldots, p_k, \ldots) &= t \Mon(1, t^{1/2}, t^{-1} p_1, t^{-3/2} p_2, \ldots, t^{-(k+1)/2} p_k, \ldots) \\
  \Mon(z, t, p_1, p_2, \ldots, p_k, \ldots) &= t^{-2} \Gen(t^2, z t^2 p_1, z^2 t^3 p_2, \ldots, z^k t^{k+1} p_k, \ldots)
\end{align}
In addition to replacing the marker for number of transpositions by a marker for genus, the marker $z$ for the size of the ground set has been removed from $\Gen$; this is mainly to simplify the task of keeping track of whether the operators $\diff{p_k}$ are considered as $z$-linear or $s$-linear operators, where $z$ and $s$ are related by a functional relation, as in \autoref{thm:goseries}. See \autoref{sec:change} for the details.

\begin{theorem}\label{thm:jcgenus}
  For $f = f(t, x, p_1, p_2, \ldots) \in \QQ[t, x, p_1, p_2, \ldots]$, the differential equation
  \begin{equation}\label{ljcg}
    f = \Pi_y \Split_{x \to y} f + t \Delta_x f + f^2 + x
  \end{equation}
  has a unique solution with no constant term, $f = \Delta_x \Gen(t, p_1, p_2, \ldots)$.
  Furthermore, the series $\Delta_x \Gen_0 \in \QQ[[x, p_1, p_2, \ldots]]$ is uniquely determined by the equation
  \begin{equation}\label{ljcgz}
    \Delta_x \Gen_0 = \Pi_y \Split_{x \to y} \Delta_x \Gen_0 + (\Delta_x \Gen_0)^2 + x
  \end{equation}
  and the requirement that it have no constant term. For $g \geq 1$, the series $\Gen_g \in \QQ[[x, p_1, p_2, \ldots]]$ is uniquely determined by the equation
  \begin{equation}\label{ljchg}
    \left( 1 - 2 \Delta_x \Gen_0 - \Split_{x \to y} \right) \Delta_x \Gen_g = \Delta_x^2 \Gen_{g-1} + \sum_{g'=1}^{g-1} \Delta_x \Gen_{g'} \, \Delta_x \Gen_{g-g'}.
  \end{equation}
\end{theorem}

\begin{proof}
  Equation~\eqref{ljcg} can be obtained directly from the recurrence in \autoref{thm:recur} as in the proof of \autoref{thm:diffeq1} by using the weight
  \begin{equation}
    \frac{t^{(r - \ell(\alpha) - d + 2) / 2} p_\alpha x^k}{\prod_{i \geq 1} i^{m_i(\alpha)} m_i(\alpha)!},
  \end{equation}
  or from \eqref{diffeqsubst} by using the substitutions \eqref{substitution}. Extracting the coefficient of $t^0$ in \eqref{ljcg} gives \eqref{ljcgz}. The uniqueness comes from the fact that \eqref{ljcg} is equivalent to the recurrence in \autoref{thm:recur}. Extracting the coefficient of $t^g$ for $g \geq 1$ gives \eqref{ljchg} after moving some terms to the left-hand side to solve for $\Delta_x \Gen_g$.
\end{proof}

Note that \eqref{ljchg} expresses the image of $\Delta_x \Gen_g$ under a $\QQ[[p_1, p_2, \ldots]]$-linear operator in terms of generating series for lower genera. Our strategy to obtain $\Delta_x \Gen_g$ (and hence $\Gen_g$) is to use \eqref{ljcgz} to verify a conjectured expression for the coefficients of $\Delta_x \Gen_0$, and then to invert the linear operator in \eqref{ljchg}.

\subsection{Topological recursion}\label{sec:toprec}

For the purposes of this section, let $\alpha = (\alpha_1, \alpha_2, \ldots, \alpha_\ell)$ be a composition of $d$ instead of a partition; that is, we still have $\alpha_1 + \alpha_2 + \cdots + \alpha_\ell = d$, but we no longer require $\alpha_1 \geq \alpha_2 \geq \cdots \geq \alpha_\ell$.
Also for this section only, for given $g \geq 0$ and $\ell \geq 1$, consider the generating series
\begin{equation}
  \Mon_g(x_1, x_2, \ldots, x_\ell) = \sum_{\alpha_1, \alpha_2, \ldots, \alpha_\ell \geq 1} \frac{\mon_g(\alpha)}{\abs{\C_\alpha}} x_1^{\alpha_1 - 1} x_2^{\alpha_2 - 1} \cdots x_\ell^{\alpha_\ell - 1}.
\end{equation}
This series for monotone Hurwitz numbers, which collects only the terms for $\alpha$ with a fixed number of parts, is analogous to the series \linebreak $H_g(x_1, x_2, \ldots, x_\ell)$ for Hurwitz numbers considered by Bouchard and Mari\~no \cite[Equations~(2.11) and~(2.12)]{BM}.

One form of recurrence for Hurwitz numbers is expressed in terms of the series $H_g(x_1, x_2, \ldots, x_\ell)$. This is sometimes referred to as ``topological recursion'' (see, \textit{e.g.}, \cite[Conjecture~2.1]{BM}; \cite[Remark~4.9]{EMS}; \cite[Definition~4.2]{EO}).

\begin{theorem}\label{thm:toprec}
  For $g \geq 0$ and $\ell \geq 1$, we have
  \begin{multline}
    \Mon_g(x_1, x_2, \ldots, x_\ell)
      = \delta_{g,0} \delta_{\ell,1}
      + x_1 \Mon_{g-1}(x_1, x_1, x_2, \ldots, x_\ell) \vphantom{\sum_{j=2}^\ell} \\
      + \sum_{j=2}^\ell \diff{x_j} \left( \frac{x_1 \Mon_g(x_1, \ldots, \widehat{x_j}, \ldots x_\ell) - x_j \Mon_g(x_2, \ldots, x_\ell)}{x_1 - x_j} \right) \\
      + \sum_{g'=0}^g \sum_{S \subseteq \{2, \ldots, k\}} x_1 \Mon_{g'}(x_1, x_S) \Mon_{g-g'}(x_1, x_{\overline{S}}),
  \end{multline}
  where $x_1, \ldots, \widehat{x_j}, \ldots x_\ell$ is the list of all variables $x_1, \ldots, x_\ell$ except $x_j$, $x_S$ is the list of all variables $x_j$ with $j \in S$, and $x_{\overline{S}}$ is the list of all variables $x_j$ with $j \in \{2, \ldots, k\} \setminus S$.
\end{theorem}

\begin{proof}
  The result follows routinely from the monotone join-cut recurrence in \autoref{thm:recur}.
\end{proof}

\begin{remark}
  Note that as a consequence of its combinatorial interpretation, the unique solution of the recurrence in \autoref{thm:toprec} is symmetric in the variables $x_1, x_2, \ldots, x_\ell$, even though the recurrence itself is asymmetric between $x_1$ and $x_2, \ldots, x_\ell$.
\end{remark}

For $(g, \ell) = (0, 1)$, the recurrence above reduces to the equation $\Mon_0(x_1) = 1 + x_1 \Mon_0(x_1)^2$, so we have
\begin{equation}
  \Mon_0(x_1) = \frac{1 - \sqrt{1 - 4x_1}}{2x_1} = \sum_{k \geq 0} \frac{1}{k + 1} \binom{2k}{k} x_1^k.
\end{equation}
After some calculation we obtain
\begin{equation}
  \Mon_0(x_1, x_2) = \frac{4}{\sqrt{1 - 4x_1} \sqrt{1 - 4x_2} (\sqrt{1 - 4x_1} + \sqrt{1 - 4x_2})^2}.
\end{equation}
If we define $y_i$ by $y_i = 1 + x_i y_i^2$ for $i \geq 1$, then
\begin{align}
  \Mon_0(x_1) &= y_1, \\
  \Mon_0(x_1, x_2) &= \frac{x_1 y_1' x_2 y_2' (x_2 y_2 - x_1 y_1)^2}{(y_1 - 1) (y_2 - 1) (x_2 - x_1)^2},
\end{align}
where $y_i'$ denotes $\diff[y_i]{x_i}.$
Thus, in the context of Eynard and Orantin \cite{EO}, we have the spectral curve $y$, where $y = 1 + xy^2$, but it is unclear to us what the correct notion of Bergmann kernel should be in this case.

\subsection{Monotone double Hurwitz numbers}\label{sec:double}

In this section, we give a combinatorial description of the boundary conditions of the monotone 
join-cut equation for monotone double Hurwitz numbers.  The monotone double Hurwitz numbers were dealt with extensively in \cite{GGN}; 
$\hur^r(\alpha,\beta)$ counts $(r+2)$-tuples $(\sigma,\rho,\tau_1,\dots,\tau_r)$
with $\rho \in C_\beta$, satisfying conditions $(1)-(5)$ of Theorem \ref{thm:Degenerate} suitably modified.

\begin{theorem}
  The generating series
  \begin{equation}
    \Mon = \Mon(z,t,p, q) = \sum_{\substack{d \geq 1 \\ r \geq 0 \\ \alpha, \beta \vdash d}} \frac{z^d}{d!} t^r p_\alpha q_\beta \mon^r(\alpha, \beta)
  \end{equation}
  for the monotone double Hurwitz numbers satisfies the equation
  \begin{multline}\label{residue}
  \frac{1}{2t} \bigg{(}  z \diff[\Mon]{z}-zp_1 q_1 - z^2 \diff{z} \sum_{\substack{i, j \geq 1 \\ d \geq i, j \\ r \geq 0 \\ \alpha \vdash d - i \\ \beta \vdash d - j}} \frac{z^d}{d!} t^r p_\alpha p_{i+1} q_\beta q_{j+1} N^r(\alpha, i; \beta, j)
     \bigg{)}\\ = \frac{1}{2} \sum_{i, j \geq 1} (i + j) p_i p_j \diff[\Mon]{p_{i+j}} + ij p_{i+j} \sdiff[\Mon]{p_i}{p_j} + ij p_{i+j} \diff[\Mon]{p_i} \diff[\Mon]{p_j},
  \end{multline}
  where $N^r(\alpha, i; \beta, j)$ is the number of transitive monotone solutions of
  \begin{equation}\label{newdouble}
    \sigma \rho (a_1 \, b_1) (a_2 \, b_2) \cdots (a_r \, b_r) = \id
  \end{equation}
  where $\sigma \in \C_{\alpha \cup \{i\}}$, $\rho \in \C_{\beta \cup \{j\}}$, and the element $d$ is in a cycle of length $i$ of $\sigma$ and a cycle of length $j$ of $\rho$.
\end{theorem}

\begin{proof}
  The proof is very similar in spirit to the proofs of Theorems~\ref{thm:recur} and \ref{thm:JoinCut}, so we only sketch the combinatorial join-cut analysis.
  
  For a fixed choice of $\sigma \in \C_{\alpha \cup \{i\}}$ with the element $d$ in a cycle of length $i$ and $\rho \in \C_{\beta \cup \{j\}}$ with $d$ in a cycle of length $j$, consider the monotone factorizations of $\rho^{-1} \sigma^{-1}$ of length $r$ which correspond to solutions of \eqref{newdouble} counted by $\mon^r(\alpha, \beta)$. These factorizations are counted with a weight of $\frac{z^d}{d!} t^r p_\alpha p_i q_\beta q_j$. They are not necessarily the transitive monotone factorizations of $\rho^{-1} \sigma^{-1}$, as the subgroup generated by the transpositions may be a proper subgroup of the subgroup generated by the transpositions and $\sigma$ and $\rho$. However, the join-cut analysis performed for \autoref{thm:recur} in the case where last factor of the monotone factorization involves the element $d$ (that is, the cases of cuts, redundant joins, and essential joins) applies here essentially unchanged anyway, and gives the second term of the left-hand side of \eqref{residue}. When doing this analysis, $\rho$ stays fixed while $\sigma$ is modified, which corresponds to the fact that the variables $q_i$ do not appear explicitly on the left-hand side of \eqref{residue}.
  
  If $d > 1$ and the element $d$ is not involved in any transpositions, then it must be a fixed point of $\sigma \rho$, so there is some element $k$ with $\rho(d) = \sigma^{-1}(d) = k$. In this case, the subgroup
  \begin{equation}
    \langle \sigma, \rho, (a_1 \, b_1), (a_2 \, b_2), \ldots, (a_r \, b_r) \rangle
  \end{equation}
  acts transitively on the ground set $\{1, 2, \ldots, d\}$ if and only if $k \neq d$ and the subgroup
  \begin{equation}
    \langle \sigma (k \, d), (k \, d) \rho, (a_1 \, b_1), (a_2 \, b_2), \ldots, (a_r \, b_r) \rangle
  \end{equation}
  acts transitively on the subset $\{1, 2, \ldots, d - 1\}$. Let $\sigma' = \sigma (k \, d)$ and $\rho' = (k \, d) \rho$. Then, $\sigma' \in \C_{\alpha \cup \{i - 1\}}$ with $k$ in a cycle of length $i - 1$, $\rho' \in \C_{\beta \cup \{j - 1\}}$ with $k$ in a cycle of length $j - 1$, and
  \begin{equation}
    \sigma' \rho' (a_1 \, b_1) (a_2 \, b_2) \cdots (a_r \, b_r) = \id
  \end{equation}
  is a transitive monotone solution of \eqref{newdouble} counted by $N^r(\alpha, i - 1; \beta, j - 1)$. This gives the second term on the right-hand side of \eqref{residue}.

  The last case is where $d = 1$, and then $\mon^r(\varepsilon, \varepsilon) = \delta_{r,0}$, which gives the first term on the right-hand side of \eqref{residue}.
\end{proof}

\section{Variables and operators}\label{sec:change}


To obtain the series given in Theorems~\ref{thm:goseries} from the join-cut differential equation, we perform a Lagrangian change of variables. That is, the generating series $\Mon_g$ are defined in terms of the variable $z$, while the expressions in these theorems are written in terms of the variable $s$, and these two variables are related by the functional relation
\begin{equation}s = z \left( 1 - \sum_{k \geq 1} \binom{2k}{k} p_k s^k  \right)^{-2}.\end{equation}
However, it is more convenient to work with the quantities $z^k p_k$ and $s^k p_k$ instead of working with $z$ and $s$ directly. Therefore, we work with the generating series $\Gen_g$, which correspond to $\Mon_g$ with $z^k p_k$ replaced by $p_k$, and introduce another set of variables, $q_1, q_2, \ldots$, to replace $s^k p_k$.

In this section, we describe the relation between the variables $p_1, p_2, \ldots$ and the variables $q_1, q_2, \ldots$ and introduce the basic power series used to express $\Gen_g$ succinctly. We also collect a few computational lemmas. The notation of this section is used for the rest of this paper.

\subsection{Two sets of variables}

Let $q_1, q_2, \ldots$ be a countable set of indeterminates, and let
\begin{equation}\gamma = \sum_{k \geq 1} \binom{2k}{k} q_k, \qquad \eta = \sum_{k \geq 1} (2k + 1) \binom{2k}{k} q_k\end{equation}
be formal power series in these indeterminates. If we set
\begin{equation}\label{pqrel}
  p_k = q_k (1 - \gamma)^{2k}
\end{equation}
for all $k \geq 1$, then $p_1, p_2, \ldots$ are power series in $q_1, q_2, \ldots$. Since these power series have no constant term, and the linear term in $p_k$ is simply $q_k$, they can be solved recursively to write $q_1, q_2, \ldots$ as power series in $p_1, p_2, \ldots$. Thus, we can identify the rings of power series in these two sets of variables, and write
\begin{equation}R = \QQ[[p_1, p_2, \ldots]] = \QQ[[q_1, q_2, \ldots]].\end{equation}
Using the multivariate Lagrange Implicit Function Theorem (see \cite[Theorem 1.2.9]{GJ:book}), we can relate the coefficient extraction operators $[p_\alpha]$ and $[q_\alpha]$ as follows.
\begin{theorem}\label{thm:LIFT}
  Let $\alpha \vdash d \geq 0$ be a partition and $f \in R$ be a power series. Then
  \begin{equation}[p_\alpha] f = [q_\alpha] f \frac{1 - \eta}{(1 - \gamma)^{2d+1}}.\end{equation}
\end{theorem}

\begin{proof}
  Let $\phi_k = (1 - \gamma)^{-2k}$, so that $q_k = p_k \phi_k$. Then, from~\cite[Theorem 1.2.9]{GJ:book}, we get
  \begin{equation}[p_\alpha] f = [q_\alpha] f \phi_\alpha \det\left( \delta_{ij} - q_j \diff{q_j} \log \phi_i \right)_{i,j \geq 1}.\end{equation}
  We have $\phi_\alpha = (1 - \gamma)^{-2d}$, and using the fact that $\det(I - AB) = \det(I - BA)$ for matrices $A$ and $B$, we can compute the determinant as
  \begin{align}
    \det\left( \delta_{ij} - q_j \diff{q_j} \log \phi_i \right)_{i,j \geq 1}
      &= \det\left( \delta_{ij} - \frac{2i q_j}{1 - \gamma} \binom{2j}{j} \right)_{i,j \geq 1} \\
      &= 1 - \sum_{k \geq 1} \frac{2k q_k}{1 - \gamma} \binom{2k}{k} \\
      &= \frac{1 - \eta}{1 - \gamma}.\qedhere
  \end{align}
\end{proof}

Let
\begin{equation}\D_k = p_k\diff{p_k}, \qquad \DD = \sum_{k \geq 1} k \D_k\end{equation}
be differential operators on $R$. Note that they all have the set $\{p_\alpha\}_{\alpha \vdash d \geq 0}$ as an eigenbasis, with eigenvalues given by
\begin{equation}\D_k p_\alpha = m_k p_\alpha, \qquad \DD p_\alpha = \abs{\alpha} p_\alpha,\end{equation}
and consequently commute with each other. Let
\begin{equation}\E_k = q_k\diff{q_k}, \qquad \EE = \sum_{k \geq 1} k \E_k\end{equation}
be the corresponding differential operators on $R$ for the basis $\{q_\alpha\}_{\alpha \vdash d \geq 0}$. This second set of differential operators commutes with itself, but not with the first. By using the relation \eqref{pqrel} and computing the action of $\E_k$ on $p_j$, we can verify that
\begin{equation}\label{etod}
  \E_k = \D_k - \frac{2q_k}{1 - \gamma} \binom{2k}{k} \DD
\end{equation}
as operators. It follows that
\begin{equation}\label{ED}
  \EE = \frac{1 - \eta}{1 - \gamma} \DD
\end{equation}
and
\begin{equation}\label{dtoe}
  \D_k = \E_k + \frac{2q_k}{1 - \eta} \binom{2k}{k} \EE,
\end{equation}
so that we can write each set of differential operators in terms of the other. Using the fact that $2\EE \gamma = \eta - \gamma$, it will also be useful to note that for any integer $i$ and any power series $A \in R$,
\begin{equation}\label{intfact}
  (1 - \gamma)^i (2\EE - i) \Big( (1 - \gamma)^i A \Big) = \frac{1 - \eta}{1 - \gamma} (2\DD - i) A,
\end{equation}
which reduces to the relation \eqref{ED} when $i = 0$.

\subsection{Projecting and splitting}

For the series $\Delta_x \Gen_g$, the change of variables from $p_1, p_2, \ldots$ to $q_1, q_2, \ldots$ also corresponds to the change of variables from $x, y, z$ to $\hat{x}, \hat{y}, \hat{z}$ defined by the relations
\begin{equation}x = \hat{x} (1 - \gamma)^2, \qquad y = \hat{y} (1 - \gamma)^2, \qquad z = \hat{z} (1 - \gamma)^2.\end{equation}
Since the power series $(1 - \gamma)^2$ is invertible in the ring $R$, we can identify the $R$-algebras $R[[x]]$ and $R[[\hat{x}]]$, and similarly for any subset of $\{x,y,z\}$ and the corresponding subset of $\{\hat{x}, \hat{y}, \hat{z}\}$. In terms of these variables, we have
\begin{equation}\Pi_x = [x^0] + \sum_{k \geq 1} p_k [x^k] = [\hat{x}^0] + \sum_{k \geq 1} q_k [\hat{x}^k],\end{equation}
and similarly for $\Pi_y$, $\Pi_z$, and
\begin{equation}\Split_{x \to y}\big( f(x) \big) = \frac{yf(x) - xf(y)}{x - y} = \frac{\hat{y}f(x) - \hat{x}f(y)}{\hat{x} - \hat{y}}\end{equation}
for a series $f(x) \in R[[x]]$ with no constant term in $x$.

\subsection{Series and polynomials}

In addition to the series
\begin{equation}\gamma = \sum_{k \geq 1} \binom{2k}{k} q_k, \qquad \eta = \sum_{k \geq 1} (2k + 1) \binom{2k}{k} q_k,\end{equation}
we define the series
\begin{equation}\eta_j = \EE^j \eta = \sum_{k \geq 1} k^j (2k + 1) \binom{2k}{k} q_k\end{equation}
for $j \geq 0$. Note that $\eta_0 = \eta$. Our solutions to the join-cut equation will be expressed in terms of these. In particular, for genus 2 and higher, $\Gen_g$ will lie in the subring
\begin{equation}Q = \QQ[(1 - \eta)^{-1}, \eta_1 (1 - \eta)^{-1}, \eta_2 (1 - \eta)^{-1}, \ldots]\end{equation}
of $R$. The further subring
\begin{equation}P = \QQ[\eta_1 (1 - \eta)^{-1}, \eta_2 (1 - \eta)^{-1}, \ldots] \subset Q\end{equation}
will also be useful, especially in the context of the $\QQ$-vector space decomposition
\begin{equation}Q = \bigoplus_{i \geq 0} (1 - \eta)^{-i} P.\end{equation}

Let $u,v,w$ be defined by
\begin{equation}u = (1 - 4\hat{x})^{-\frac{1}{2}}, \qquad v = (1 - 4\hat{y})^{-\frac{1}{2}}, \qquad w = (1 - 4\hat{z})^{-\frac{1}{2}}.\end{equation}
Then, we have
\begin{equation}u = \sum_{k \geq 0} \binom{2k}{k} \hat{x}^k, \quad\text{and}\quad \hat{x}\diff{\hat{x}} = \frac{u^3 - u}{2} \diff{u},\end{equation}
so
\begin{equation}\gamma = \Pi_x (u - 1), \quad
  \eta = \Pi_x (u^3 - 1), \quad
  \eta_1 = \Pi_x \tfrac{3}{2} (u^5 - u^3), \quad
  \ldots\end{equation}
In fact, if we define
\begin{equation}\eta_j^u = \left(\hat{x}\diff{\hat{x}}\right)^j (u^3 - 1),\end{equation}
so that $\eta_j = \Pi_x \eta_j^u$, then it can be seen that for $j \geq 1$, $\eta_j^u$ is an odd polynomial in $u$ of degree $2j + 3$ divisible by $(u^5 - u^3)$. Thus, the set $\{\eta_j^u\}_{j \geq 1}$ is a $\QQ$-basis for the vector space $(u^5 - u^3)\QQ[u^2]$. This will be useful to show that various expressions project down to the subring $Q$ of $R$, or to a particular subspace of $Q$.

\subsection{Lifting}

With this notation, it is useful to compute the image of the lifting operator $\Delta_x$ on some elements of $R[[\hat{x}, \hat{y}, \hat{z}]]$.
\begin{lemma}\label{lem:comp}
\begin{align}
  \Delta_x(q_k) &= k\hat{x}^k + \frac{k(u^3 - u)q_k}{1 - \eta} &
  \Delta_x(\gamma) &= \frac{(u^3 - u)(1 - \gamma)}{2(1 - \eta)} \\
  \Delta_x(\hat{x}) &= \frac{\hat{x}(u^3 - u)}{1 - \eta} &
  \Delta_x(u) &= \frac{(u^3 - u)^2}{2(1 - \eta)} \\
  \Delta_x(\hat{y}) &= \frac{\hat{y}(u^3 - u)}{1 - \eta} &
  \Delta_x(v) &= \frac{(u^3 - u)(v^3 - v)}{2(1 - \eta)} \\
  \Delta_x(\eta_j) &= \eta_{j+1}^u + \frac{(u^3 - u) \eta_{j+1}}{1 - \eta}.
\end{align}
\end{lemma}

From this list, we can also conclude that $\Delta_x$ preserves some important subspaces of $Q[u]$.
\begin{corollary}\label{cor:deltasubspace}
For each $i, j \geq 0$, the operator $\Delta_x$ restricts to a function
\begin{equation}\Delta_x \colon (u^3 - u)^i (1 - \eta)^{-j} P[u^2] \to (u^3 - u)^{i+1} (1 - \eta)^{-j-1} P[u^2].\end{equation}
\end{corollary}

\section{Genus zero}\label{sec:genus-zero}


In this section, we prove Theorem \ref{thm:gzformula}.
Our strategy is to define the series $\Gen_0$ by
\begin{equation}\label{gggdef}
  \Gen_0 = \sum_{d \geq 1} \sum_{\alpha \vdash d} \sum_{r \geq 0} (2d + 1) (2d + 2) \cdots (2d + \ell - 3) \frac{|C_\alpha|}{d!}p_\alpha \prod_{k = 1}^\ell \alpha_k \binom{2\alpha_k}{\alpha_k}
\end{equation}
and then to show that the series $\Delta_x \Gen_0$ satisfies the genus zero join-cut equation \eqref{ljcgz} of \autoref{thm:jcgenus}. This involves performing a Lagrangian change of variables to get a closed form for $\Delta_x \Gen_0$.

\begin{theorem}
  With the definition above for $\Gen_0$,
  \begin{equation}(2\DD - 2) (2\DD - 1) (2\DD) \Gen_0 = \frac{(1 - \gamma)^3}{1 - \eta} - 1.\end{equation}
\end{theorem}

\begin{proof}
  For $\alpha \vdash d \geq 0$, $p_\alpha$ is an eigenvector of the operator $\DD$ with eigenvalue $d$, so we can use this operator to transform the expression for $\Gen_0$ into a negative binomial. For $\alpha \vdash d \geq 1$, using \autoref{thm:LIFT}, we have
  \begin{multline}
    [p_\alpha] (2\DD - 2) (2\DD - 1) (2\DD) \Gen_0 \\
    \begin{aligned}
      &= (2d - 2) (2d - 1) (2d) (2d + 1) \cdots (2d + \ell - 3) \frac{\abs{\C_\alpha}}{d!}
         \prod_{k = 1}^\ell \alpha_k \binom{2\alpha_k}{\alpha_k} \\
      &= (-1)^\ell \binom{2 - 2d}{\ell} \binom{\ell}{m_1, m_2, \ldots}
         \prod_{j \geq 1} \binom{2j}{j}^{m_j} p_j^{m_j} \\
      &= [q_\alpha] (1 - \gamma)^{2 - 2d} \\
      &= [p_\alpha] \frac{(1 - \gamma)^3}{1 - \eta}.
    \end{aligned}
  \end{multline}

  \begin{remark}
    This formula is the main motivation for the definition of $s$ through the functional relation
    \begin{equation}s = z \left( 1 - \sum_{k \geq 1} \binom{2k}{k} s^k p_k \right)^{-2}\end{equation}
    and the change of variables from $z$ to $s$ or, equivalently, from $p_1, p_2, \ldots$ to $q_1, q_2, \ldots$.
  \end{remark}

  After computing the coefficient of $p_\alpha$ when $d = 0$ separately, we get
  \begin{equation}(2\DD - 2) (2\DD - 1) (2\DD) \Gen_0 = \frac{(1 - \gamma)^3}{1 - \eta} - 1.\qedhere\end{equation}
\end{proof}

\begin{theorem}
  Let
  \begin{equation}\Gen_0' = (2\DD) \Gen_0, \qquad \Gen_0'' = (2\DD - 1) \Gen_0', \qquad \Gen_0''' = (2\DD - 2) \Gen_0''.\end{equation}
  Then,
  \begin{align}
    \label{ggg2} \Gen_0'' &= \tfrac{1}{2} - \tfrac{1}{2} (1 - \gamma)^2, \\
    \label{dggg2} \D_k \Gen_0'' &= \frac{(1 - \gamma)^2}{1 - \eta} \binom{2k}{k} q_k, \\
    \label{dggg1} \D_k \Gen_0' &= (1 - \gamma) \frac{1}{2k - 1} \binom{2k}{k} q_k, \\
    \label{dggg} \D_k \Gen_0 &= \frac{1}{2k(2k - 1)} \binom{2k}{k} q_k - \sum_{j \geq 1} \frac{2j + 1}{2(j + k)(2k - 1)} \binom{2j}{j} \binom{2k}{k} q_j q_k.
  \end{align}
\end{theorem}

\begin{proof}
  By \eqref{intfact} with $i = 2$, we have
  \begin{align}
    \Gen_0''
      &= (2\DD - 2)^{-1} \Gen_0''' \\
      &= \tfrac{1}{2} + (2\DD - 2)^{-1} (\Gen_0''' + 1) \\
      &= \tfrac{1}{2} + (1 - \gamma)^2 (2\EE - 2)^{-1} 1 \\
      &= \tfrac{1}{2} - \tfrac{1}{2} (1 - \gamma)^2.
  \end{align}
  Note that the kernel of $(2\DD - 2)$ is spanned by $p_\alpha$ where $d = 1$, that is, $p_1$. Since the constant and linear terms for power series in $p_1, p_2, \ldots$ and series in $q_1, q_2, \ldots$ are equal, it is easy to check that this expression for $\Gen_0''$ agrees with the definition of $\Gen_0$ on the coefficient of $p_1$. This establishes \eqref{ggg2}.
  Using the relation \eqref{ED} to convert between $\DD$ and $\EE$, we then have
  \begin{equation}\D_k \Gen_0'' = (1 - \gamma) \D_k \gamma
                  = \frac{(1 - \gamma)^2}{1 - \eta} \binom{2k}{k} q_k,\end{equation}
  which establishes \eqref{dggg2}.
  
  Since the operator $\D_k$ commutes with $(2\DD - 1)$, we can use \eqref{intfact} again, this time with $i = 1$, to get
  \begin{align}
    \D_k \Gen_0'
      &= (2\DD - 1)^{-1} \D_k \Gen_0'' \\
      &= (1 - \gamma) (2\EE - 1)^{-1} \binom{2k}{k} q_k \\
      &= (1 - \gamma) \frac{1}{2k - 1} \binom{2k}{k} q_k.
  \end{align}
  This establishes \eqref{dggg1}.
  
  Finally, since $\D_k$ also commutes with $(2\DD)$ and we know that $\D_k \Gen_0$ has no constant term, by \eqref{intfact} with $i = 0$, we have
  \begin{align}
    \D_k \Gen_0
      &= (2\DD)^{-1} \D_k \Gen_0' \\
      &= (2\EE)^{-1} (1 - \eta) \frac{1}{2k - 1} \binom{2k}{k} q_k \\
      &= (2\EE)^{-1} \frac{1}{2k - 1} \binom{2k}{k} q_k - \sum_{j \geq 1} \frac{2j + 1}{2k - 1} \binom{2j}{j} \binom{2k}{k} q_j q_k \\
      &= \frac{1}{2k(2k - 1)} \binom{2k}{k} q_k - \sum_{j \geq 1} \frac{2j + 1}{2(j + k)(2k - 1)} \binom{2j}{j} \binom{2k}{k} q_j q_k.
  \end{align}
  This establishes \eqref{dggg}.
\end{proof}

With this expression for the partial derivatives of $\Gen_0$, we can compute $\Delta_x \Gen_0$. To get a concise expression, we introduce the following power series.

\begin{lemma}\label{lem:fseries}
  The power series
  \begin{equation}F(\hat{x}, \hat{y}) = \sum_{j \geq 0} \sum_{k \geq 1} \frac{(2j + 1)k}{2(j + k)(2k - 1)} \binom{2j}{j} \binom{2k}{k} \hat{x}^k \hat{y}^j\end{equation}
  can be expressed as
  \begin{equation}F(\hat{x}, \hat{y}) = \frac{(u^2 - 1) v^2}{2u(u + v)}.\end{equation}
\end{lemma}

\begin{proof}
  \begin{align}
    F(\hat{x}, \hat{y})
      &= \sum_{j \geq 0} \sum_{k \geq 1} \frac{(2j + 1)k}{2(j + k)(2k - 1)} \binom{2j}{j} \binom{2k}{k} \hat{x}^k \hat{y}^j \\
      &= \int_0^1 \hat{x}t (1 - 4\hat{x}t)^{-\frac{1}{2}} (1 - 4\hat{y}t)^{-\frac{3}{2}} \frac{dt}{t} \\
      &= \left[ \frac{\hat{x}}{2(\hat{y} - \hat{x})} (1 - 4\hat{x}t)^{\frac{1}{2}} (1 - 4\hat{y}t)^{-\frac{1}{2}} \right]_{t = 0}^1 \\
      &= \frac{\hat{x}}{2(\hat{y} - \hat{x})} \Big( (1 - 4\hat{x})^{\frac{1}{2}} (1 - 4\hat{y})^{-\frac{1}{2}} - 1 \Big) \\
      &= \frac{(u^2 - 1) v^2}{2u(u + v)}. \qedhere
  \end{align}
\end{proof}

\begin{theorem}
	\label{thm:gzspec}
  If $\Gen_0$ is defined by \eqref{gggdef}, then
  \begin{equation}\label{www}
    \Delta_x \Gen_0 = 2F(\hat{x}, 0) - \Pi_y F(\hat{x}, \hat{y}),
  \end{equation}
  where $F(\hat{x}, \hat{y})$ is as defined in \autoref{lem:fseries}, and this series satisfies the genus zero join-cut equation \eqref{ljcgz} of \autoref{thm:jcgenus}. Therefore, $\Gen_0$ is the generating series for genus zero monotone single Hurwitz numbers.
\end{theorem}

\begin{proof}
  By \eqref{dggg}, we have
  \begin{align}
    \Delta_x \Gen_0
      &= \sum_{k \geq 1} \frac{k\hat{x}^k}{q_k} \D_k \Gen_0 \\
      &= \sum_{k \geq 1} \frac{1}{2(2k - 1)} \binom{2k}{k} \hat{x}^k - \sum_{j \geq 1} \sum_{k \geq 1} \frac{(2j + 1)k}{2(j + k)(2k - 1)} \binom{2j}{j} \binom{2k}{k} 
      q_j \hat{x}^k \\
      &= 2F(\hat{x}, 0) - \Pi_y F(\hat{x}, \hat{y}).
  \end{align}

  To verify that $\Delta_x \Gen_0$ satisfies equation \eqref{ljcgz}, we need to check that the expression
  \begin{equation}
    \Delta_x \Gen_0 - \Pi_y \Split_{x \to y} \Delta_x \Gen_0 - (\Delta_x \Gen_0)^2 - x
  \end{equation}
  is zero. We can rewrite each of these terms as
  \begin{align}
    \Delta_x \Gen_0 &= \Pi_{yz} \Big( 2F(\hat{x}, 0) - F(\hat{x}, \hat{y}) \Big), \\
    \Pi_y \Split_{x \to y} \Delta_x \Gen_0 &= \Pi_{yz} \left( \frac{\hat{y}\big( 2F(\hat{x}, 0) - F(\hat{x}, \hat{z}) \big) - \hat{x}\big( 2F(\hat{y}, 0) - F(\hat{y}, \hat{z}) \big)}{\hat{x} - \hat{y}} \right), \\
    (\Delta_x \Gen_0)^2 &= \Pi_{yz} \Big( \big( 2F(\hat{x}, 0) - F(\hat{x}, \hat{y}) \big) \big( 2F(\hat{x}, 0) - F(\hat{x}, \hat{z}) \big) \Big), \\
    x &= \hat{x} (1 - \gamma)^2 = \Pi_{yz} \Big( \tfrac{1}{4}(1 - u^{-2}) (2 - v) (2 - w) \Big),
  \end{align}
  to get an expression of the form
  \begin{equation}\Pi_{yz} W(u, v, w),\end{equation}
  where $W(u, v, w)$ is a rational function of $u, v, w$. This rational function itself is not zero, but a straightforward computation shows that its symmetrization with respect to $y$ and $z$, that is, $\tfrac{1}{2} \big( W(u, v, w) + W(u, w, v) \big)$, is zero. Thus,
  \begin{equation}\Pi_{yz} W(u, v, w) = \Pi_{yz} \tfrac{1}{2} \big( W(u, v, w) + W(u, w, v) \big) = 0,\end{equation}
  which completes the verification.
\end{proof}

\noindent
{\em Proof of Theorem \ref{thm:gzformula}}.
The result follows immediately from Theorem \ref{thm:gzspec} and \eqref{gggdef}.$\hfill\Box$

\medskip

To compute the right-hand side of the join-cut equation \eqref{ljchg} for genus one, we also need the following corollary.

\begin{corollary}\label{cor:dwww}
  \begin{equation}
    \Delta_x^2 \Gen_0 = \tfrac{1}{16} (u^2 - 1)^2.
  \end{equation}
\end{corollary}

\begin{proof}
  This follows from \eqref{www} by applying the identity \eqref{deltapi} and simplifying.
\end{proof}

\section{Higher genera}\label{sec:higher-genus}


In this section, we prove Theorem \ref{thm:goseries}, giving expressions for $\Mon_g$ for $g \geq 1$. We do so by solving the higher genus join-cut equation \eqref{ljchg} of \autoref{thm:jcgenus} and keeping track of the form of the solution. This gives an expression for the generating series $\Delta_x \Gen_g$, which we then integrate to obtain $\Gen_g$, or equivalently, $\Mon_g$. To establish the value of the constant of integration which must be used, we use a formula of \cite{MN} for the one-part case of the monotone single Hurwitz numbers.



To establish degree bounds on the solutions, we will need a notion of degree on the spaces $Q$ and $(u^3 - u)Q[u^2]$. Recall from \autoref{sec:change} that $Q = \QQ[(1 - \eta)^{-1}, \eta_1 (1 - \eta)^{-1}, \eta_2 (1 - \eta)^{-1}, \ldots]$ and $(u^3 - u)Q[u^2]$ has basis over $Q$ given by $(u^3 - u), \eta_1^u, \eta_2^u, \ldots$.
\begin{definition}
  For an element
  \begin{equation}A = \sum_{\substack{\alpha \vdash d \geq 0 \\ j \geq 0}} a_{j,\alpha} \frac{\eta_\alpha}{(1 - \eta)^{j+\ell}} \in Q,\end{equation}
  its \emph{weight} is
  \begin{equation}\nu(A) = \max\{ d \colon a_{j,\alpha} \neq 0 \}.\end{equation}
  For an element
  \begin{equation}B = \sum_{\substack{\alpha \vdash d \geq 0 \\ j \geq 0}} \left( b_{j,\alpha,0} (u^3 - u) + \sum_{k \geq 1} b_{j,\alpha,k} \eta_k^u \right) \frac{\eta_\alpha}{(1 - \eta)^{j+\ell}} \in (u^3 - u)Q[u^2],\end{equation}
  its \emph{weight} is
  \begin{equation}\nu(B) = \max\{ d + k \colon b_{j,\alpha,k} \neq 0 \}.\end{equation}
\end{definition}
In particular, note that
\begin{equation}\nu\big(\eta_k (1 - \eta)^{-1}\big) = \nu(\eta_k^u) = \nu(u^{2k+3} - u^{2k+1}) = k\end{equation}
for $k \geq 1$, and
\begin{equation}\nu\big((1 - \eta)^{-1}\big) = \nu(u^3 - u) = 0,\end{equation}
so that the weight of a polynomial in $u$ can be determined from the weights of the coefficients of $u^3, u^5, u^7, \ldots$.

When solving the higher genus equation, we will also need the $R$-linear operator $\T$ defined by
\begin{align}
  \T \colon (u^3 - u) R[u^2] &\to (u^3 - u) R[u^2] \\
  (u^3 - u) f(u^2) &\mapsto \frac{u^3 - u}{1 - \eta} \Pi_y \frac{(v^5 - v^3) (f(u^2) - f(v^2))}{u^2 - v^2}.
\end{align}

\begin{lemma}\label{lem:tprop}
  The operator $\T$ is locally nilpotent, meaning that for every $(u^3 - u) f(u^2) \in (u^3 - u) R[u^2]$, there is some $n \geq 0$ with $\T^n \big( (u^3 - u) f(u^2) \big) = 0$. Furthermore, for each $j \geq 0$, the operator $\T$ restricts to a function
  \begin{equation}\T \colon (u^3 - u) (1 - \eta)^{-j} P[u^2] \to (u^3 - u) (1 - \eta)^{-j} P[u^2],\end{equation}
  and on these subspaces, $\nu\Big(\T\big((u^3 - u) f(u^2)\big)\Big) \leq \nu\big((u^3 - u) f(u^2)\big)$.
\end{lemma}

\begin{proof}
  The operator $\T$ strictly reduces the degree in $u$ of every nonzero element of $(u^3 - u) R[u^2]$, so repeated application of $\T$ will always eventually produce zero.
  
  Now, suppose that $(u^3 - u) f(u^2) \in (u^3 - u) \QQ[u^2]$, where $f(u^2)$ has degree $2k$ in $u$, so that $\nu\big((u^3 - u) f(u^2)\big) = k$. The expression $(v^5 - v^3) (f(u^2) - f(v^2)) / u^2 - v^2$ can be written as a polynomial in $u$, where the coefficient of $u^{2i}$ is a polynomial in $(v^5 - v^3) (1 - \eta)^{-j} \QQ[v^2]$ of degree at most $2k - 2i + 3$ in $v$. Then, as noted in \autoref{sec:change}, the coefficient of $u^{2i}$ is a linear combination of $\eta_1^v, \eta_2^v, \ldots, \eta_{k-i}^v$ with coefficients in $\QQ$. Applying $\Pi_y$ and multiplying by $(u^3 - u) (1 - \eta)^{-1}$ gives a linear combination of terms of the form $(u^{2i+3} - u^{2i+1}) \eta_j (1 - \eta)^{-1}$ for $j = 1, 2, \ldots, k - i$, which lie in $(u^3 - u) P[u^2]$ and have weight at most $k$.
  
  From this, the result follows by $Q$-linearity.
\end{proof}

\begin{theorem}\label{thm:wwwform}
  For $g \geq 1$,
  \begin{equation}
    \Delta_x \Gen_g \in (u^3 - u) (1 - \eta)^{1 - 2g} P[u^2]
  \end{equation}
  and $\nu(\Delta_x \Gen_g) \leq 3g - 2$. In particular,
  \begin{equation}
    \Delta_x \Gen_1 = \frac{u^5 - 2u^3 + u}{16(1 - \eta)} + \frac{(u^3 - u) \eta_1}{24(1 - \eta)^2}
  \end{equation}
\end{theorem}

\begin{proof}
  Recall that the higher genus join-cut equation \eqref{ljchg} is
  \begin{equation}
    \left( 1 - 2 \Delta_x \Gen_0 - \Split_{x \to y} \right) \Delta_x \Gen_g = \Delta_x^2 \Gen_{g-1} + \sum_{g'=1}^{g-1} \Delta_x \Gen_{g'} \, \Delta_x \Gen_{g-g'}.
  \end{equation}
  Note that the left-hand side operator is $R$-linear. Using \eqref{www} with \autoref{lem:fseries} and expressing $\Split_{x \to y}$ in terms of $u$ and $v$, we have
  \begin{multline}
    \left( 1 - 2 \Delta_x \Gen_0 - \Split_{x \to y} \right)\big( (u^3 - u) f(u^2) \big) \\
    \begin{aligned}
      &= \Pi_y\left( (2 - v^3) (u^2 - 1) f(u^2) - \frac{(u^2 - 1)(v^5 - v^3)(f(u^2) - f(v^2))}{u^2 - v^2} \right) \\
      &= \frac{(1 - \eta)}{u} (1 - \T) \big( (u^3 - u) f(u^2) \big)
    \end{aligned}
  \end{multline}
  for $f(u^2) \in R[u^2]$. Since $\T$ is locally nilpotent, it follows that this operator is invertible on $(u^2 - 1) R[u^2]$, with inverse given by
  \begin{multline}
    \left( 1 - 2 \Delta_x \Gen_0 - \Split_{x \to y} \right)^{-1}\big( (u^2 - 1) f(u^2) \big) \\
    = (1 + \T + \T^2 + \cdots) \left( \frac{(u^3 - u) f(u^2)}{1 - \eta} \right).
  \end{multline}
  Thus, we have
  \begin{equation}
    \Delta_x \Gen_g = \frac{1}{1 - \eta} (1 + \T + \T^2 + \cdots)\left( u\Delta_x^2 \Gen_{g-1} + \sum_{g'=1}^{g-1} u\Delta_x \Gen_{g'} \, \Delta_x \Gen_{g-g'} \right).
  \end{equation}
  
  Then, using \autoref{cor:dwww} for the value of $\Delta_x^2 \Gen_0$, we can compute $\Delta_x \Gen_1$ directly. Using and \autoref{cor:deltasubspace} and \autoref{lem:tprop}, it follows that $\Delta_x \Gen_g \in (u^3 - u) (1 - \eta)^{1 - 2g} P[u^2]$ by induction on $g$.
  
  Using \autoref{lem:comp}, a straightforward computation shows that for $A, B \in (u^3 - u) Q[u^2]$, we have
  \begin{equation}\nu\big(u \Delta(A)\big) \leq \nu(A) + 3, \qquad \nu(u A B) \leq \nu(A) + \nu(B) + 2,\end{equation}
  and since $\T$ does not increase weights, the bound $\nu(\Delta_x \Gen_g) \leq 3g - 2$ also follows by induction on $g$.
\end{proof}

Having solved for $\Delta_x \Gen_g = \Delta_x \Gen_g$, we now need to solve for $\Gen_g$ and show that it has the right form. To do so, we first compute $\sum_{k \geq 1} \E_k \Gen_g$ from $\Delta_x \Gen_g$, and then invert the operator $\sum_{k \geq 1} \E_k$.

\begin{theorem}\label{thm:egggform}
  We have
  \begin{equation}\sum_{k \geq 1} \E_k \Gen_1 = \frac{\eta}{24(1 - \eta)} + \frac{\gamma}{8(1 - \gamma)},\end{equation}
  and for $g \geq 2$,
  \begin{equation}\sum_{k \geq 1} \E_k \Gen_g \in (1 - \eta)^{-2} Q.\end{equation}
\end{theorem}

\begin{proof}
  It follows from \eqref{etod} that
  \begin{align}
    \sum_{k \geq 1} \E_k \Gen_g
      &= \left(\sum_{k \geq 1} \D_k - \frac{2\gamma}{1 - \gamma} \DD\right) \Gen_g \\
      &= \Pi_x\left( \left(x\diff{x}\right)^{-1} - \frac{2\gamma}{1 - \gamma} \right) \Delta_x \Gen_g \\
      &= \Pi_x\left( \left(x\diff{x}\right)^{-1} - \frac{2\gamma}{1 - \gamma} \right) \Delta_x \Gen_g.
  \end{align}
  We can replace the term $\Pi_x \left( \frac{2\gamma}{1 - \gamma} \right) \Delta_x \Gen_g$ in this expression by using the operator identity
  \begin{equation}\Pi_x = [u^0] \left( 1 - 2 \Delta_x \Gen_0 - \Split_{x \to y} \right) - (1 - \gamma) [u^1],\end{equation}
  which can be checked on the $Q$-linear space $(u^3 - u)Q[u^2]$ by verifying it on the basis $(u^3 - u), \eta_1^u, \eta_2^u, \ldots$. Then, we have
  \begin{align}
    \sum_{k \geq 1} \E_k \Gen_g
      &= \Pi_x\left( \left(x\diff{x}\right)^{-1} - 2(u - 1) [u^1] \right) \Delta_x \Gen_g \\
      &{} - \frac{2\gamma}{1 - \gamma} [u^0] \left( 1 - 2 \Delta_x \Gen_0 - \Split_{x \to y} \right) \Delta_x \Gen_g.
  \end{align}
  For $g = 1$, we can use the expression for $\Delta_x \Gen_1$ from \autoref{thm:wwwform} to compute
  \begin{equation}\sum_{k \geq 1} \E_k \Gen_1 = \frac{\eta}{24(1 - \eta)} + \frac{\gamma}{8(1 - \gamma)}\end{equation}
  from this. For $g \geq 2$, we have
  \begin{align}
    \left( 1 - 2 \Delta_x \Gen_0 - \Split_{x \to y} \right) \Delta_x \Gen_g
      &= \Delta_x^2 \Gen_{g-1} + \sum_{g'=1}^{g-1} \Delta_x \Gen_{g'} \, \Delta_x \Gen_{g-g'} \\
      &\in (u^3 - u)^2 (1 - \eta)^{2-2g} P[u^2],
  \end{align}
  so in particular, this has no constant term as a polynomial in $u$. Thus, for $g \geq 2$,
  \begin{equation}\sum_{k \geq 1} \E_k \Gen_g = \Pi_x\left( \left(x\diff{x}\right)^{-1} - 2(u - 1) [u^1] \right) \Delta_x \Gen_g.\end{equation}
  Also, we have
  \begin{align}
    \Pi_x \left( \left(x\diff{x}\right)^{-1} - 2(u - 1) [u^1] \right)(u^3 - u) &= 0, \\
    \Pi_x \left( \left(x\diff{x}\right)^{-1} - 2(u - 1) [u^1] \right) \eta_j^u &= \eta_{j-1},
  \end{align}
  for $j \geq 1$. This determines the action of this $Q$-linear operator on $(u^3 - u)Q[u^2]$. Since
  \begin{equation}
    \Delta_x \Gen_g \in (u^3 - u) (1 - \eta)^{-3} Q[u^2]
  \end{equation}
  for $g \geq 2$ and
  \begin{equation}\frac{\eta_0}{1 - \eta} = 1 - \frac{1}{1 - \eta},\end{equation}
  we get
  \begin{equation}\sum_{k \geq 1} \E_k \Gen_g \in (1 - \eta)^{-2} Q.\qedhere\end{equation}
\end{proof}

\begin{theorem}\label{thm:gggform}
  We have
  \begin{equation}\Gen_1 = \tfrac{1}{24} \log (1 - \eta)^{-1} - \tfrac{1}{8} \log (1 - \gamma)^{-1},\end{equation}
  and for $g \geq 2$,
  \begin{equation}\Gen_g = -c_{g,(0)} + \frac{1}{(1 - \eta)^{2g - 2}} \sum_{d = 0}^{3g - 3} \sum_{\alpha \vdash d} \frac{c_{g,\alpha} \eta_\alpha}{(1 - \eta)^\ell},\end{equation}
  where $c_{g,(0)}=-\frac{B_{2g}}{4g(g-1)}.$
\end{theorem}

\begin{proof}
  Note that $\gamma, \eta, \eta_1, \eta_2, \ldots$ are all eigenvectors of the differential operator $\sum_{k \geq 1} \E_k$ with eigenvalue 1, since they are purely linear in $q_1, q_2, \ldots$. Thus, up to a constant, we have
  \begin{align}
    \left( \sum_{k \geq 1} \E_k \right)^{-1} \frac{\eta}{1 - \eta}
      &= \int_0^1 \frac{\eta t}{1 - \eta t} \frac{dt}{t} = \log (1 - \eta)^{-1} \\
    \left( \sum_{k \geq 1} \E_k \right)^{-1} \frac{\gamma}{1 - \gamma}
      &= \int_0^1 \frac{\gamma t}{1 - \gamma t} \frac{dt}{t} = \log (1 - \gamma)^{-1}
  \end{align}
  Together with the constraint that $\Gen_1$ has no constant term, we get
  \begin{equation}\Gen_1 = \tfrac{1}{24} \log (1 - \eta)^{-1} - \tfrac{1}{8} \log (1 - \gamma)^{-1}.\end{equation}
  When $j \geq 2$, we have
  \begin{equation}\left( \sum_{k \geq 1} \E_k \right)^{-1} \frac{\eta_\alpha}{(1 - \eta)^{j+\ell}} = \int_0^1 \frac{\eta_\alpha t^\ell}{(1 - \eta t)^{j+\ell}} \frac{dt}{t} \in (1 - \eta)^{-1} Q,\end{equation}
  so for some constant $c_g$, for $g \geq 2$, we get
  \begin{equation}\Gen_g - c_g \in (1 - \eta)^{-1} Q.\end{equation}
  In fact, since the kernel of $\Delta_x$ on $Q$ is $\QQ$ and $\Delta_x \Gen_g \in (u^3 - u) (1 - \eta)^{1 - 2g} P[u^2]$, it follows from \autoref{cor:deltasubspace} that $\Gen_g - c_g \in (1 - \eta)^{2 - 2g} P$. Since $\nu(\Delta_x \Gen_g) \leq 3g - 2$, it follows from \autoref{lem:comp} that $\nu(\Gen_g) \leq 3g - 3$. Thus, we can write
  \begin{equation}\Gen_g = c_g + \frac{1}{(1 - \eta)^{2g - 2}} \sum_{d = 0}^{3g - 3} \sum_{\alpha \vdash d} \frac{c_{g,\alpha} \eta_\alpha}{(1 - \eta)^\ell},\end{equation}
  where the coefficients $c_{g,\alpha}$ are rational numbers. All that remains is to compute the value of $c_g$. The only other term of the sum which contributes a constant term is $c_{g,\varepsilon} (1 - \eta)^{2 - 2g}$, so we must have $c_g = -c_{g,\varepsilon}$. If we expand $\Gen_g$ as a power series in $p_1, p_2, \ldots$ and keep only the constant and linear terms, we get
  \begin{align}
    \Gen_g
      &= c_g + c_{g,\varepsilon} + (2g - 2) c_{g,\varepsilon} \eta + \sum_{k = 1}^{N_g} c_{g,(k)} \eta_k + \bigoh \\
      &= \sum_{d \geq 1} \left( (2g - 2) c_{g,\varepsilon} + \sum_{k = 1}^{N_g} c_{g,(k)} d^k \right) (2d + 1) \binom{2d}{d} p_d + \bigoh.
  \end{align}
  For fixed $g \geq 2$, the expression
  \begin{equation}f(d) = (2g - 2) c_{g,\varepsilon} + \sum_{k = 1}^{N_g} c_{g,(k)} d^k\end{equation}
  is a polynomial in $d$, and $f(0) / (2 - 2g) = c_g$.
  
  By \cite{MN}, the number of transitive monotone factorizations of genus $g$ of the cycle $(1 \, 2 \, \ldots \, d)$, or any other cycle of length $d$, is
  \begin{equation}\frac{1}{d} \binom{2d-2}{d-1} \binom{2d-2 + 2g}{2d-2} \left[\frac{z^{2g}}{(2g)!}\right] \left(\frac{\sinh(z/2)}{z/2}\right)^{2d-2}.\end{equation}
  There are $(d-1)!$ of these long cycles, so the coefficient of $p_d$ in $\Gen_g$ is
  \begin{equation}\frac{(d-1)!}{d!} \frac{1}{d} \binom{2d-2}{d-1} \binom{2d-2 + 2g}{2d-2} \left[\frac{z^{2g}}{(2g)!}\right] \left(\frac{\sinh(z/2)}{z/2}\right)^{2d-2}.\end{equation}
  Comparing this to the expression above and expanding the binomial coefficients into factorials gives
  \begin{equation}\frac{f(d)}{2 - 2g} = \frac{(2d - 2 + 2g)!}{(2 - 2g) (2d + 1) (2d)! (2g)!} \left[\frac{z^{2g}}{(2g)!}\right] \left(\frac{\sinh(z/2)}{z/2}\right)^{2d-2}.\end{equation}
  For fixed $g$, the coefficient of $z^{2g}/(2g)!$ extracted here is a polynomial in $d$, and setting $d = 0$, we get
  \begin{align}
    c_g = \frac{f(0)}{2 - 2g}
      &= \frac{1}{4g(g-1)(1-2g)} \left[\frac{z^{2g}}{(2g)!}\right] \left(\frac{\sinh(z/2)}{z/2}\right)^{-2} \\
      &= \frac{1}{4g(g-1)(1-2g)} \left[\frac{z^{2g}}{(2g)!}\right] \frac{z^2 e^z}{(e^z - 1)^2} \\
      &= \frac{1}{4g(g-1)(1-2g)} \left[\frac{z^{2g}}{(2g)!}\right] \left(1 - z\diff{z}\right) \frac{z}{e^z - 1} \\
      &= \frac{1}{4g(g-1)(1-2g)} \cdot (1 - 2g) B_{2g} \\
      &= \frac{B_{2g}}{4g(g-1)},
  \end{align}
  since $z / (e^z - 1)$ is the generating series for the Bernoulli numbers.
\end{proof}

\medskip
\noindent
{\em Proof of Theorem \ref{thm:goseries}}.
The result follows immediately from Theorem~\ref{thm:gggform} and Definition~\ref{def:Gseries}.$\hfill\Box$

\medskip

\section{Lagrange inversion and polynomiality}\label{sec:lagrange}

In this section, we use Lagrange inversion to derive an explicit formula for the genus one monotone single Hurwitz numbers (\autoref{thm:goformula}) from the generating series given in \autoref{thm:goseries}, and to establish a polynomiality result for higher genera (\autoref{thm:polynomial}).


For $k \geq 1$, let $\Theta_k \colon R \to R$ be the differential operator
\begin{equation}\Theta_k = \binom{2k}{k}^{-1} \diff{q_k}.\end{equation}
Then, we have
\begin{equation}\Theta_k(\gamma) = 1, \qquad \Theta_k(\eta) = 2k + 1, \qquad \Theta_k(\eta_i) = (2k + 1) k^i.\end{equation}
The following lemma shows how this operator is related to the quantities described in Theorems~\ref{thm:goformula} and~\ref{thm:polynomial}.

\begin{lemma}\label{lem:coex}
  If $f \in R$ is a power series and $\alpha \vdash d \geq 0$, then
  \begin{equation}\frac{d! [q_\alpha] f}{\abs{\C_\alpha} \prod_{j=1}^\ell \alpha_j \binom{2\alpha_j}{\alpha_j}} = \Theta_\alpha \left. f \right|_{q_1 = q_2 = \cdots = 0}.\end{equation}
\end{lemma}

\begin{proof}
  Since
  \begin{equation}\frac{d!}{\abs{\C_\alpha}} = \prod_{j=1}^\ell \alpha_j \cdot \prod_{k \geq 1} m_k!,\end{equation}
  where $m_1, m_2, \ldots$ are the part multiplicities of $\alpha$, we have
  \begin{equation}\frac{d! [q_\alpha] f}{\abs{\C_\alpha} \prod_{j=1}^\ell \alpha_j \binom{2\alpha_j}{\alpha_j}}
    = \prod_{j=1}^\ell \binom{2\alpha_j}{\alpha_j}^{-1} \left[ \frac{q_1^{m_1}}{m_1!} \frac{q_2^{m_2}}{m_2!} \cdots \right] f
    = \Theta_\alpha \left. f \right|_{q_1 = q_2 = \cdots = 0}.\qedhere\end{equation}

\end{proof}

\begin{proof}[Proof of \autoref{thm:goformula}]
  To show that
  \begin{equation}\frac{24 \mon_1(\alpha)}{\abs{\C_\alpha} \prod_{i = 1}^\ell \alpha_i \binom{2\alpha_i}{\alpha_i}} = (2d + 1)^{\overline{\ell}} - 3 (2d + 1)^{\overline{\ell - 1}} - \sum_{k = 2}^\ell (k - 2)! (2d + 1)^{\overline{\ell - k}} e_k(2\alpha + 1),\end{equation}
  we apply Lagrange inversion to the generating series
  \begin{equation}\Gen_1 = \tfrac{1}{24} \log (1 - \eta)^{-1} - \tfrac{1}{8} \log (1 - \gamma)^{-1}\end{equation}
  from \autoref{thm:gggform}, expand the result as a power series in $\gamma$ and $\eta$, and apply $\Theta_\alpha$ to each term.
  By \autoref{thm:LIFT}, we have
  \begin{multline}
    [p_\alpha] \log (1 - \eta)^{-1} \\
    \begin{aligned}
      &= [q_\alpha] \log (1 - \eta)^{-1} (1 - \eta) (1 - \gamma)^{-2d-1} \\
      &= [q_\alpha] \left( \eta - \sum_{k \geq 2} \frac{\eta^k}{k(k - 1)} \right) \left( \sum_{j \geq 0} \binom{-2d-1}{j} (-1)^j \gamma^j \right) \\
      &= [q_\alpha] \left( \eta - \sum_{k \geq 2} (k - 2)! \frac{\eta^k}{k!} \right) \left( \sum_{j \geq 0} (2d + 1)^{\overline{j}} \frac{\gamma^j}{j!} \right) \\
      &= [q_\alpha] \left( (2d + 1)^{\overline{\ell-1}} \eta \frac{\gamma^{\ell-1}}{(\ell-1)!} - \sum_{k = 2}^\ell (k - 2)! (2d + 1)^{\overline{\ell-k}} \frac{\eta^k}{k!} \frac{\gamma^{\ell-k}}{(\ell-k)!} \right).
    \end{aligned}
  \end{multline}
  Note that, by iterating the product rule, we get
  \begin{align}
    \Theta_\alpha \left( \frac{\eta^k}{k!} \frac{\gamma^{\ell-k}}{(\ell-k)!} \right)
      &= \sum_{1 \leq i_1 < \cdots < i_k \leq \ell} \frac{\Theta_{\alpha_{i_1}}(\eta) \cdots \Theta_{\alpha_{i_k}}(\eta)}{\Theta_{\alpha_{i_1}}(\gamma) \cdots \Theta_{\alpha_{i_k}}(\gamma)} \Theta_{\alpha_1}(\gamma) \cdots \Theta_{\alpha_\ell}(\gamma) \\
      &= \sum_{1 \leq i_1 < \cdots < i_k \leq \ell} (2\alpha_{i_1} + 1) (2\alpha_{i_2} + 1) \cdots (2\alpha_{i_k} + 1) \\
      &= e_k(2\alpha + 1),
  \end{align}
  and this is constant as a power series in $q_1, q_2, \ldots$, and hence unaffected by evaluation at $q_1 = q_2 = \cdots = 0$. Using \autoref{lem:coex}, we have
  \begin{multline}
    \frac{d! [p_\alpha] \log (1 - \eta)^{-1}}{\abs{\C_\alpha} \prod_{i = 1}^\ell \alpha_i \binom{2\alpha_i}{\alpha_i}} \\
    \begin{aligned}
      &= \Theta_\alpha \left( (2d + 1)^{\overline{\ell-1}} \eta \frac{\gamma^{\ell-1}}{(\ell-1)!} - \sum_{k = 2}^\ell (k - 2)! (2d + 1)^{\overline{\ell-k}} \frac{\eta^k}{k!} \frac{\gamma^{\ell-k}}{(\ell-k)!} \right) \\
      &= (2d + 1)^{\overline{\ell-1}} e_1(2\alpha + 1) - \sum_{k = 2}^\ell (k - 2)! (2d + 1)^{\overline{\ell-k}} e_k(2\alpha + 1) \\
      &= (2d + 1)^{\overline{\ell}} - \sum_{k = 2}^\ell (k - 2)! (2d + 1)^{\overline{\ell-k}} e_k(2\alpha + 1),
    \end{aligned}
  \end{multline}
  since $e_1(2\alpha + 1) = (2d + \ell)$. This gives most of the terms on the right-hand side of the genus one formula.
  
  For the remaining term, using \autoref{thm:LIFT} again and the fact that $p_\alpha, q_\alpha$ are eigenvectors for the operators $\DD, \EE$ respectively, we have
  \begin{align}
    [p_\alpha] \log (1 - \gamma)^{-1}
      &= \tfrac{1}{d} [p_\alpha] \DD \log (1 - \gamma)^{-1} \\
      &= \tfrac{1}{d} [q_\alpha] (1 - \gamma)^{-2d} \frac{1 - \eta}{1 - \gamma} \DD \log (1 - \gamma)^{-1} \\
      &= \tfrac{1}{d} [q_\alpha] (1 - \gamma)^{-2d} \EE \log (1 - \gamma)^{-1} \\
      &= \tfrac{1}{d} [q_\alpha] \EE \Big( \tfrac{1}{2d} (1 - \gamma)^{-2d} \Big) \\
      &= [q_\alpha] \tfrac{1}{2d} (1 - \gamma)^{-2d} \\
      &= [q_\alpha] (2d + 1)^{\overline{\ell-1}} \frac{\gamma^\ell}{\ell!}.
  \end{align}
  Applying \autoref{lem:coex}, we get
  \begin{align}
    \frac{d! [p_\alpha] \log (1 - \gamma)^{-1}}{\abs{\C_\alpha} \prod_{i = 1}^\ell \alpha_i \binom{2\alpha_i}{\alpha_i}}
      &= \Theta_\alpha \left( (2d + 1)^{\overline{\ell-1}} \frac{\gamma^\ell}{\ell!} \right) \\
      &= (2d + 1)^{\overline{\ell-1}}.
  \end{align}
  This gives the remaining term on the right-hand side of the genus one formula.
\end{proof}

Using similar techniques, we can get a polynomiality result for monotone single Hurwitz numbers for fixed genus and number of parts.

\begin{proof}[Proof of \autoref{thm:polynomial}]
  Noting that $d = \sum_{j=1}^\ell \alpha_j$, the explicit formulas for genus zero and one from Theorems~\ref{thm:gzformula} and~\ref{thm:goformula} show that the expression
  \begin{equation}\frac{\mon_g((\alpha_1, \alpha_2, \ldots, \alpha_\ell))}{\abs{\C_\alpha} \prod_{i=1}^\ell \alpha_i \binom{2\alpha_i}{\alpha_i}}\end{equation}
  is a polynomial in $\alpha_1, \alpha_2, \ldots, \alpha_\ell$ when $g = 0$ and $\ell \geq 3$, and when $g = 1$ and $\ell \geq 1$. For $g \geq 2$, we have the rational form
  \begin{equation}\Gen_g = -c_{g,(0)} + \frac{1}{(1 - \eta)^{2g - 2}} \sum_{d = 0}^{3g - 3} \sum_{\alpha \vdash d} \frac{c_{g,\alpha} \eta_\alpha}{(1 - \eta)^\ell}.\end{equation}
  As in the proof of \autoref{thm:goformula}, we can apply Lagrange inversion to each of the terms, expand the result as power series in $\gamma, \eta, \eta_1, \eta_2, \ldots$, and extract coefficients using \autoref{lem:coex}.
  
  By \autoref{thm:LIFT}, we have
  \begin{equation}[p_\alpha] \frac{\eta_1^{a_1} \eta_2^{a_2} \cdots \eta_k^{a_k}}{(1 - \eta)^j} = [q_\alpha] \frac{\eta_1^{a_1} \eta_2^{a_2} \cdots \eta_k^{a_k}}{(1 - \eta)^{j-1} (1 - \gamma)^{2d + 1}}.\end{equation}
  This can be expanded as an infinite linear combination of terms of the form
  \begin{equation}[q_\alpha] \frac{\eta_0^{a_0} \eta_1^{a_1} \cdots \eta_k^{a_k}}{(1 - \gamma)^{2d + 1}},\end{equation}
  but only the finitely many terms with $a_0 + a_1 + \cdots + a_k \leq \ell$ have a nonzero contribution. For these terms, applying \autoref{lem:coex} gives
  \begin{multline}
    \frac{d! [q_\alpha]}{\abs{\C_\alpha} \prod_{i = 1}^\ell \alpha_i \binom{2\alpha_i}{\alpha_i}} \left( \frac{\eta_0^{a_0} \cdots \eta_k^{a_k}}{(1 - \gamma)^{2d + 1}} \right) \\
    \begin{aligned}
      &= \frac{d! [q_\alpha]}{\abs{\C_\alpha} \prod_{i = 1}^\ell \alpha_i \binom{2\alpha_i}{\alpha_i}} \left( (2d + 1)^{\overline{\ell - \sum_i a_i}} \eta_0^{a_0} \cdots \eta_k^{a_k} \frac{\gamma^{\ell - \sum_i a_i}}{(\ell - \sum_i a_i)!} \right) \\
      &= (2d + 1)^{\overline{\ell - \sum_i a_i}} \Theta_\alpha \left(\eta_0^{a_0} \cdots \eta_k^{a_k} \frac{\gamma^{\ell - \sum_i a_i}}{(\ell - \sum_i a_i)!} \right),
    \end{aligned}
  \end{multline}
  which is a symmetric polynomial in $\alpha_1, \alpha_2, \ldots, \alpha_\ell$. For fixed $\ell$,
  \begin{equation}\frac{\mon_g((\alpha_1, \alpha_2, \ldots, \alpha_\ell))}{\abs{\C_\alpha} \prod_{i=1}^\ell \alpha_i \binom{2\alpha_i}{\alpha_i}} = \frac{d! [p_\alpha] \Gen_g}{\abs{\C_\alpha} \prod_{i = 1}^\ell \alpha_i \binom{2\alpha_i}{\alpha_i}}\end{equation}
  is a finite linear combination of these polynomials, so it is a polynomial in $\alpha_1, \alpha_2, \ldots, \alpha_\ell$.
\end{proof}

\appendix
\section{Rational forms for $g=2$ and $3$}\label{sec:forms}


The following equations give the rational forms for the genus two and three generating series for the monotone single Hurwitz numbers, as described in \autoref{thm:goseries}.

Genus two:
\begin{equation}720 \vec{\mathbf{H}}_2 = -3 + \frac{3}{(1 - \eta)^2} + \frac{5 \eta_3 - 6 \eta_2 - 5 \eta_1}{(1 - \eta)^3} + \frac{29 \eta_2 \eta_1 - 10 \eta_1^2}{(1 - \eta)^4} + \frac{28 \eta_1^3}{(1 - \eta)^5}.\end{equation}

Genus three:
\begin{multline}
90720 \vec{\mathbf{H}}_3 = 90 + \frac{-90}{(1 - \eta)^4}
+ \frac{70 \eta_6 + 63 \eta_5 - 377 \eta_4 - 189 \eta_3 + 667 \eta_2 + 126 \eta_1}{(1 - \eta)^5} \\
{} + \frac{1078 \eta_1 \eta_5 + 2012 \eta_2 \eta_4 + 1209 \eta_1 \eta_4 + 1214 \eta_3^2}{(1 - \eta)^6} \\
{} + \frac{1998 \eta_2 \eta_3 - 3914 \eta_1 \eta_3 - 2627 \eta_2^2 - 2577 \eta_1 \eta_2 + 1967 \eta_1^2}{(1 - \eta)^6} \\
{} + \frac{8568 \eta_1^2 \eta_4 + 26904 \eta_1 \eta_2 \eta_3 + 10092 \eta_1^2 \eta_3 + 5830 \eta_2^3}{(1 - \eta)^7} \\
{} + \frac{13440 \eta_1 \eta_2^2 - 20322 \eta_1^2 \eta_2 - 4352 \eta_1^3}{(1 - \eta)^7} \\
{} + \frac{44520 \eta_1^3 \eta_3 + 86100 \eta_1^2 \eta_2^2 + 49980 \eta_1^3 \eta_2 - 15750 \eta_1^4}{(1 - \eta)^8} \\
{} + \frac{162120 \eta_1^4 \eta_2 + 31080 \eta_1^5}{(1 - \eta)^9}
+ \frac{68600 \eta_1^6}{(1 - \eta)^{10}}.
\end{multline}

\bibliographystyle{amsplain}

\end{document}